\documentclass[a4paper,12pt,reqno]{amsart}

\usepackage{amsmath,amssymb,amsthm}
\usepackage{latexsym}
\usepackage{color}
\usepackage{graphicx}
\usepackage{mathrsfs}
\usepackage{enumerate}
\usepackage[abbrev]{amsrefs}
\usepackage[T1]{fontenc}
\usepackage{mathtools}
\mathtoolsset{showonlyrefs=true}

\allowdisplaybreaks[4]

\theoremstyle{plain}
\newtheorem{thm}{Theorem}[section]
\newtheorem{lemm}[thm]{Lemma}

\newtheorem{cor}[thm]{Corollary}

\theoremstyle{definition}
\newtheorem{df}[thm]{Definition}
\newtheorem{rem}[thm]{Remark}

\makeatletter

\@addtoreset{equation}{section}
\makeatother

\renewcommand{\div}{\operatorname{div}}
\newcommand{\dB}{\dot{B}}
\newcommand{\supp}{\operatorname{supp}}

\newcommand{\n}[1]{{\left\|#1\right\|}}

\begin{document}
\title[$2$D stationary Navier--Stokes equations on the whole plane]
{Ill-posedness of the two-dimensional stationary Navier--Stokes equations on the whole plane}
\author[M.~Fujii]{Mikihiro Fujii}
\address[M.~Fujii]{Institute of Mathematics for Industry, Kyushu University, Fukuoka 819--0395, Japan}
\email{fujii.mikihiro.096@m.kyushu-u.ac.jp}
\begin{abstract}
    We consider the two-dimensional stationary Navier--Stokes equations on the whole plane $\mathbb{R}^2$.
    In the higher-dimensional cases $\mathbb{R}^n$ with $n \geqslant 3$, the well-posedness and ill-posedness in scaling critical spaces are well-investigated by numerous papers.
    However, despite the attention of many researchers, the corresponding problem in the two-dimensional whole plane case was a long-standing open problem due to inherent difficulties of two-dimensional analysis.
    The aim of this paper is to address this issue and prove the ill-posedness in the scaling critical Besov spaces based on $L^p(\mathbb{R}^2)$ for all $1 \leqslant p \leqslant2$ in the sense of the discontinuity of the solution map and the non-existence of small solutions.
    To overcome the difficulty, we propose a new method based on the contradictory argument that reduces the problem to the analysis of the corresponding nonstationary Navier--Stokes equations and shows the existence of nonstationary solutions with strange large time behavior, if we suppose to contrary that the stationary problem is well-posed.
\end{abstract}
\maketitle

\section{Introduction}\label{sec:intro}
We consider the incompressible stationary Navier--Stokes equations on $\mathbb{R}^n$ with $n \geqslant 2$:
\begin{align}\label{eq:sNS-0}
    \begin{cases}
    -\Delta U + (U \cdot \nabla ) U + \nabla P = F, \qquad & x \in \mathbb{R}^n,\\
    \div U = 0 , \qquad  & x \in \mathbb{R}^n,
    \end{cases}
\end{align}
where $U=U(x):\mathbb{R}^n \to \mathbb{R}^n$ and $P=P(x):\mathbb{R}^n \to \mathbb{R}$ denote the unknown velocity fields and unknown pressure of the fluid, respectively, whereas $F=F(x):\mathbb{R}^n \to \mathbb{R}^n$ is the given external force.
In the higher-dimensional cases $\mathbb{R}^n$ with $n \geqslant 3$, the well-posedness and ill-posedness in the scaling critical framework are well-investigated (see \cites{Che-93,Kan-Koz-Shi-19,Koz-Yam-95-PJA,Koz-Yam-95-IUMJ,Li-Yu-Zhu,Tsu-19-JMAA,Tsu-19-ARMA,Tsu-19-DIE}).
Although these fundamental problems in two-dimensional case have attracted the attention of many researchers, it has remained unsolved until now because of the difficulties inherent in two-dimensional analysis.
In the present paper, we address this open problem and prove that the stationary Navier--Stokes equations on $\mathbb{R}^2$ is ill-posed in the scaling critical Besov spaces based on $L^p(\mathbb{R}^2)$ for all $1 \leqslant p \leqslant 2$.

Before stating our main result precisely, we reformulate the problem, define the concepts of well-posedness and ill-posedness, and then review the previous studies related to our problem.
Let $\mathbb{P}:= I + \nabla \div (-\Delta)^{-1} = \left\{ \delta_{jk} + \partial_{x_j}\partial_{x_k}(-\Delta)^{-1} \right\}_{1 \leqslant j,k \leqslant n}$
be the Helmholtz projection onto the divergence-free vector fields.
Applying $(-\Delta)^{-1}\mathbb{P}$ to the equation \eqref{eq:sNS-0} and using the facts $\mathbb{P}U=U$, $(U \cdot \nabla)U=\div(U \otimes U)$, and $\mathbb{P}(\nabla P)=0$, we see that \eqref{eq:sNS-0} is formally equivalent to 
\begin{align}\label{eq:sNS-2}
    U = (-\Delta)^{-1}\mathbb{P}F - (-\Delta)^{-1}\mathbb{P}\div (U \otimes U),\qquad x \in \mathbb{R}^n.
\end{align}
For a Banach space $S \subset \mathscr{S}'(\mathbb{R}^n)$,
we say that $U \in S$ is a solution to \eqref{eq:sNS-0} if $U$ satisfies \eqref{eq:sNS-2} in $S$.
Next, we define the notion of well-posedness and ill-posedness.
\begin{df}\label{df-WP}
For two Banach spaces $D,S \subset \mathscr{S}'(\mathbb{R}^n)$,
we say that the equation \eqref{eq:sNS-0} is well-posed from the data space $D$ to the solution space $S$ if the following three statements hold:
\begin{itemize}
    \item [(i)] There exists a positive constant $\delta$ such that for any $F \in B_D(\delta)$, \eqref{eq:sNS-0} possesses a solution $U \in S$,
    \item [(ii)] There exists a positive constant $\varepsilon$ such that the solution of \eqref{eq:sNS-0} is unique in the class $B_S(\varepsilon)$,
    \item [(iii)] The solution map $B_D(\delta) \ni F \mapsto U \in B_S(\varepsilon)$, which is well-defined by (i) and (ii), is continuous, 
\end{itemize}
where we have set $B_D(\delta):=\{ F \in D\ ;\ \| F \|_D < \delta\}$ and $B_S(\varepsilon):=\{ U \in S \ ;\ \| U \|_S < \varepsilon \}$.
If \eqref{eq:sNS-0} is {\it not} well-posed from $D$ to $S$, we say that the equation \eqref{eq:sNS-0} is ill-posed  from $D$ to $S$.
\end{df}
Since the pioneering work \cite{Fuj-Kat-64} by Fujita--Kato, it has been well-known that considering the well-posedness and ill-posedness in the critical function spaces with respect to scaling transforms that keep the equations invariant is crucial.
If $(F,U)$ satisfies \eqref{eq:sNS-2}, then the scaled functions
\begin{align}
    F_{\lambda}(x):= \lambda^3 F( \lambda x),\qquad
    U_{\lambda}(x):= \lambda U(\lambda x)
\end{align}
also solve \eqref{eq:sNS-2} for all $\lambda >0$. We call that the data space $D$ and the solution space $S$ are scaling critical if 
\begin{align}\label{scale-norm}
    \| F_{\lambda} \|_{D} = \| F \|_{D},\qquad
    \| U_{\lambda} \|_{S} = \| U \|_{S}
\end{align}
for all $\lambda >0$.
As the homogeneous Besov spaces $D=\dB_{p,q}^{\frac{n}{p}-3}(\mathbb{R}^n)$ and $S=\dB_{p,q}^{\frac{n}{p}-1}(\mathbb{R}^n)$ ($1 \leqslant p ,q \leqslant \infty$) satisfy \eqref{scale-norm} for all dyadic numbers $\lambda>0$, we regard them as the scaling critical Besov spaces for \eqref{eq:sNS-0}.

Next, we recall known results related to our study.
In the higher-dimensional cases $\mathbb{R}^n$ with $n \geqslant 3$, 
Leray \cite{Ler-33}, Ladyzhenskaya \cite{Lad-59}, and Fujita \cite{Fuj-61} proved the existence of solutions to \eqref{eq:sNS-0}.
For the scaling critical framework, Chen \cite{Che-93} proved the well-posedness of \eqref{eq:sNS-0} from $F=\div \widetilde{F}$ with $\widetilde{F} \in L^{\frac{n}{2}}(\mathbb{R}^n)$ to $U \in L^n(\mathbb{R}^n)$.
Kozono--Yamazaki \cites{Koz-Yam-95-PJA, Koz-Yam-95-IUMJ} considered the well-posedness and stability in the scaling critical Morrey spaces.
Kaneko--Kozono--Shimizu \cite{Kan-Koz-Shi-19} proved that \eqref{eq:sNS-0} is well-posed from $\dB_{p,q}^{\frac{n}{p}-3}(\mathbb{R}^n)$ to $\dB_{p,q}^{\frac{n}{p}-1}(\mathbb{R}^n)$ for all $(p,q) \in [1,n) \times [1, \infty]$, 
whereas
Tsurumi \cites{Tsu-19-JMAA,Tsu-19-ARMA} showed the ill-posedness for $(p,q) \in (\{ n \} \times (2,\infty])\cup((n,\infty] \times [1, \infty])$.
Li--Yu--Zhu \cite{Li-Yu-Zhu} considered the remaining case $(p,q) \in \{n\} \times [1,2]$.
For other related results,
see Tsurumi \cite{Tsu-19-DIE} for the well-posedness in the scaling critical Triebel--Lizorkin spaces, 
Tsurumi \cite{Tsu-19-N} for the well-posedness and ill-posedness in the scaling critical Besov spaces on the periodic box $\mathbb{T}^n$ ($n \geqslant 3$), 
and Cunanan--Okabe--Tsutsui \cite{Cun-Oka-Tsu-22}, Heywood \cite{Hey-70}, and Kozono--Shimizu \cite{Koz-Shi-23} for the asymptotic stability around the stationary flow.

In the two-dimensional case $n=2$, 
the following boundary value problem in exterior domains $\Omega$ with the smooth boundary have been studied extensively.
\begin{align}\label{exterior}
    \begin{dcases}
    -\Delta U + (U \cdot \nabla ) U + \nabla P = F, \qquad & x \in \Omega,\\
    \div U = 0, \qquad & x \in \Omega,\\
    U=0, \qquad & x \in \partial \Omega,\\
    \lim_{|x| \to \infty} U(x)=U_{\infty},
    \end{dcases}
\end{align}
where $U_{\infty} \in \mathbb{R}^2$ is a given constant vector.
The cause of the difference between the two-dimensional case and the higher-dimensional cases follows from the difference of the behavior at $|x| \to \infty$ for the fundamental solution $\Theta=\Theta(x)$ to the Stokes operator $-\mathbb{P}\Delta=-\Delta$ on $\mathbb{R}^n$:
\begin{align}\label{fundamental}
    \Theta(x)=
    \begin{dcases}
        -
        \frac{1}{2\pi}\log|x|,\qquad & (n=2),\\
        \frac{1}{n(n-2)\omega_n}|x|^{-(n-2)},\qquad & (n\geqslant 3),
    \end{dcases}
\end{align}
where $\omega_n$ denotes the volume of the unit ball in $\mathbb{R}^n$.
This is closely related to {\it the Stokes paradox} that the Stokes equations, which is the linearization of \eqref{exterior}, has no solution.
Chang--Finn \cite{Cha-Fin-61} showed the Stokes paradox rigorously. 
The Stokes paradox implies that it is unable to construct solutions of the Navier--Stokes equations as a perturbation \eqref{exterior} from the Stokes flow.
In contrast, Finn--Smith \cite{Fin-Smi-67-1} considered the linearized equation of the perturbed system for \eqref{exterior} around the constant flow $U_{\infty} \in\mathbb{R}^2 \setminus \{ 0 \}$ and showed that the fundamental solution of
the Oseen operator $-\Delta U + (U_{\infty} \cdot \nabla)U + \nabla P$ decays as $|x| \to \infty$ due to the term $(U_{\infty} \cdot \nabla)U$.
Finn--Smith \cite{Fin-Smi-67-2} used this fact and constructed the two-dimensional Navier--Stokes flow on exterior domains around a sufficiently small constant vector $U_{\infty} \in \mathbb{R}^2 \setminus\{0\}$ with no external force. 
This result was improved by many studies; see \cites{Ami-84,Pli-Rus-12,Gal-Soh-95,Yam-16} for instance.
We should note that the problem becomes hard in the case of $U_{\infty}=0$ since the Oseen operator coincides with the Stokes operator in this case. 
Yamazaki \cite{Yam-16} considered this case and proved the existence of a unique small solutions provided that the domain, the external force, and the solution are invariant  under the action of the cyclic group of order 4. 
For other studies on the exterior domain case with $U_{\infty}=0$, see \cites{Hil-Wit-13} for the stationary solutions around the large swirling flow $\mu x^{\perp}/|x|^2$ ($|\mu| \gg 1$) and see \cite{Mae-17} for the asymptotic stability around small swirling flows.
We refer to Galdi \cite{Gal-11} for more detail information of \eqref{exterior}.

In the whole plane case $\mathbb{R}^2$, the previous studies are fewer than for the exterior domain case. 
Indeed, it is more difficult to construct stationary solutions than the exterior domain case since
the singularity at $x=0$ as well as the increase as $|x|\to \infty$ of the fundamental solution $\Theta$ must be controlled. 
Yamazaki \cite{Yam-09} made use of some symmetric structures and constructed small solution. In \cite{Yam-09}, he considered \eqref{exterior} in the whole plane case $\Omega=\mathbb{R}^2$ with $U_{\infty}=0$ and proved that for given external force $F=\nabla^{\perp}G=(\partial_{x_2}G,-\partial_{x_1}G)$, where $G$ decays like $|G(x)|\leqslant \delta(1+|x|)^{-2}$ with some $0<\delta \ll 1$ and possesses the following symmetric conditions:
\begin{align}\label{G-sym}
    G(-x_1,x_2)=G(x_1,-x_2)=G(x_2,x_1)=G(-x_2,x_1)=-G(x_1,x_2),
\end{align}
there exists a unique small solution to \eqref{eq:sNS-0} in the $L^{2,\infty}(\mathbb{R}^2)$-framework with the vorticity $\operatorname{rot}U$ satisfying the same condition as for $G$. 
In related studies, Galdi--Yamazaki \cite{Gal-Yam-15} showed the stability of the above solutions.
See \cite{Mae-Tsu-23} for the stationary solution on $\mathbb{R}^2$ around the small swirling flow $\mu x^{\perp}/|x|^2$ ($|\mu|\ll1$).

Despite of numerous studies on the two-dimensional stationary Navier--Stokes equations, it was a long-standing open problem whether the two-dimensional Navier--Stokes equations on both the exterior domains and the whole plane $\mathbb{R}^2$ possesses a unique small solution for a given small external force $F$ in general settings without any symmetric condition.
In particular, unlike the higher-dimensional cases,
the well-posedness and ill-posedness of stationary Navier--Stokes equations on the whole plane case in the scaling critical framework were completely unsolved.

The aim of this paper is to solve the aforementioned open problem in the challenging case $\mathbb{R}^2$ and prove the ill-posedness of the two-dimensional stationary Navier--Stokes equations
\begin{align}\label{eq:sNS-1}
    \begin{cases}
    -\Delta U + (U \cdot \nabla ) U + \nabla P = F, \qquad & x \in \mathbb{R}^2,\\
    \div U = 0 , \qquad  & x \in \mathbb{R}^2
    \end{cases}
\end{align}
from the scaling critical Besov spaces $\dB_{p,1}^{\frac{2}{p}-3}(\mathbb{R}^2)$ to $\dB_{p,1}^{\frac{2}{p}-1}(\mathbb{R}^2)$ for all $1 \leqslant p \leqslant 2$.
Our main result of this paper now reads as follows.
\begin{thm}[Ill-posedness of \eqref{eq:sNS-1}]\label{thm:main}
For any $1 \leqslant p \leqslant 2$, \eqref{eq:sNS-1} is ill-posed from $\dB_{p,1}^{\frac{2}{p}-3}(\mathbb{R}^2)$ to $\dB_{p,1}^{\frac{2}{p}-1}(\mathbb{R}^2)$ in the sense that the solution map is discontinuous.
More precisely, for any $1 \leqslant p \leqslant 2$, there exist a positive constant $\delta_0=\delta_0(p)$, a positive integer $N_0=N_0(p)$, and a sequence $\{ F_N \}_{N \in \mathbb{N}} \subset \dB_{p,1}^{\frac{2}{p}-3}(\mathbb{R}^2)$ satisfying
\begin{align}
    \lim_{N \to \infty} \| F_N \|_{\dB_{p,1}^{\frac{2}{p}-3}}
    =0
\end{align}
such that
if each $F_N$ with $N \geqslant N_0$ generates a solution $U_N \in \dB_{p,1}^{\frac{2}{p}-1}(\mathbb{R}^2)$ of \eqref{eq:sNS-1}, then it holds 
\begin{align}\label{main:u-0}
    \inf_{N \geqslant N_0}
    \| U_N \|_{\dB_{p,1}^{\frac{2}{p}-1}}
    \geqslant
    \delta_0.
\end{align}
\end{thm}
\begin{rem}
We provide some remarks on Theorem \ref{thm:main}.
\begin{enumerate}
    \item 
    In the context of ill-posedness, the narrower function spaces framework, the stronger the result.
    Besov spaces with the interpolation index $q=1$ enable us to handle narrower space than Lebesgue or Sobolev spaces.
    Indeed, Theorem \ref{thm:main} includes the narrowest scaling critical Besov spaces framework from $\dB_{1,1}^{-1}(\mathbb{R}^2)$ to $\dB_{1,1}^1(\mathbb{R}^2)$,
    which are included in all scaling critical Lebesgue, Sobolev, and Besov spaces. 
    Moreover, the scaling critical Besov spaces $\dB_{p,1}^{\frac{2}{p}-1}(\mathbb{R}^2)$ ($1 \leqslant p \leqslant 2$) ensures the unconditional uniqueness for the nonstationary Navier--Stokes equations \eqref{eq:nNS-4} below, which plays a key role in the proof of Theorem \ref{thm:main}. See the outline of the proof below and Section \ref{sec:pf} for details.
    These are reasons why we use Besov spaces.
    \item 
    Theorem \ref{thm:main} can be compared with the result of Yamazaki \cite{Yam-09}, 
    where he constructed a unique small solution to \eqref{eq:sNS-1} in the scaling critical space $L^{2,\infty}(\mathbb{R}^2)$, which is a wider framework than ours, that is $\dB_{p,1}^{\frac{2}{p}-1}(\mathbb{R}^2) \hookrightarrow L^{2,\infty}(\mathbb{R}^2)$ ($1 \leqslant p \leqslant 2$).
    In \cite{Yam-09}, it is assumed that the small external force has the form $F=\nabla^{\perp}G$ with some function $G$ satisfying the symmetric condition \eqref{G-sym}, while our sequence of external forces in Theorem \ref{thm:main} is given by an anisotropic form as follows:
    \begin{align}\label{rem:force}
        F_N(x):= \frac{\delta}{\sqrt{N}}\nabla^{\perp}(\Psi(x)\cos(Mx_1)),
    \end{align}
    for some constants $0<\delta\ll 1$, $M \gg 1$, and some real valued radial symmetric function $\Psi \in \mathscr{S}(\mathbb{R}^2)$.
    Therefore, it is revealed that the symmetric condition \eqref{G-sym} is a crucial assumption for the solvability of \eqref{eq:sNS-1}. 
    \item
    In the higher-dimensional whole space $\mathbb{R}^n$ and periodic box $\mathbb{T}^n$ cases with $n \geqslant 3$, 
    it was shown in \cites{Kan-Koz-Shi-19,Tsu-19-N} that \eqref{eq:sNS-0} is well-posed in the scaling critical Besov spaces based on $L^p(\mathbb{R}^n)$ for $1 \leqslant p < n$.
    Tsurumi \cite{Tsu-23} revealed that similar results hold for the two-dimensional stationary Navier--Stokes equations on the periodic box $\mathbb{T}^2$. 
    In \cite{Tsu-23}, he showed the well-posed in the {nearly} scaling critical Besov spaces based on $L^{p+\varepsilon}(\mathbb{T}^2)$ for $1 \leqslant p < 2$ with small $\varepsilon>0$.
    By comparing these results and Theorem \ref{thm:main}, we see that, unlike the higher-dimensional cases, the solvability is different in the two-dimensional case when the domain is the periodic box $\mathbb{T}^2$ and the whole plane $\mathbb{R}^2$.
    This implies that in the two-dimensional case, information at the spatial infinity of \eqref{eq:sNS-1} affects the solvability of \eqref{eq:sNS-1}, which may be attributed to the fact that the fundamental solution $\Theta(x)$ of the two-dimensional Stokes equations increases logarithmically (see \eqref{fundamental}).
\end{enumerate}
\end{rem}
Since uniqueness is not guaranteed, there may be several solution sequences  for a fixed sequence  $\{ F_N\}_{N\in \mathbb{N}}$ of external forces.
Theorem \ref{thm:main} claims that there exists no solution to \eqref{eq:sNS-1} in $\dB_{p,1}^{\frac{2}{p}-1}(\mathbb{R}^2)$ for some $F_{N_0}$, or {\it all} sequences of solutions are bounded from below by a positive constant $\delta_0$, which is independent of the choice of solution sequences.
This implies the non-existence of small solutions for some small external forces.
More precisely, Theorem \ref{thm:main} immediately leads the following corollary.
\begin{cor}[Non-existence of small solutions to \eqref{eq:sNS-1}]\label{cor:main}
    For any $1 \leqslant p \leqslant 2$, there exist two positive constants $\delta_0=\delta_0(p)$ and  $\varepsilon_0=\varepsilon_0(p)$ such that 
    for any $0 < \varepsilon \leqslant \varepsilon_0$, there exists a external force $F^{\varepsilon} \in \dB_{p,1}^{\frac{2}{p}-3}(\mathbb{R}^2)$ satisfying $\| F^{\varepsilon} \|_{\dB_{p,1}^{\frac{2}{p}-3}}< \varepsilon$ such that 
    \eqref{eq:sNS-1} with the external force $F^{\varepsilon}$ possesses no solution in the class
    \begin{align}
        \left\{
        U \in \dB_{p,1}^{\frac{2}{p}-1}(\mathbb{R}^2) \ ;\ 
        \| U \|_{\dB_{p,1}^{\frac{2}{p}-1}} < \delta_0
        \right\}.
    \end{align}
\end{cor}

We elaborate upon the difficulty that we meet when we prove Theorem \ref{thm:main}.
Following the standard ill-posedness argument as proposed in \cites{Bou-Pav-08,Tsu-19-JMAA,Yon-10},
we may construct a sequence $\{F_N\}_{N\in \mathbb{N}} \subset \dB_{p,q}^{\frac{2}{p}-3}(\mathbb{R}^2)$ of the external force satisfying 
\begin{align}
    \lim_{N\to\infty} \| F_N\|_{\dB_{p,q}^{\frac{2}{p}-3}} = 0, \qquad
    \lim_{N\to\infty} \left\| U_N^{(1)} \right\|_{\dB_{p,q}^{\frac{2}{p}-1}} = 0, \qquad 
    \liminf_{N \to \infty} \left\| U_N^{(2)} \right\|_{\dB_{p,q}^{\frac{2}{p}-1}} > 0,
\end{align}
where $U_N^{(1)}$ and $U_N^{(2)}$ are the first and second iterations, respectively, defined as 
\begin{align}
    U_N^{(1)}:= (-\Delta)^{-1}\mathbb{P}F_N,\qquad
    U_N^{(2)}
    := 
    -
    (-\Delta)^{-1}
    \mathbb{P} 
    \div 
    ( U_N^{(1)} \otimes U_N^{(1)} ) .
\end{align}
We formally decompose the corresponding solution $U_N$ of \eqref{eq:sNS-0} with the external force $F_N$ as 
\begin{align}
    U_N = U_N^{(1)} + U_N^{(2)} + W_N,
\end{align}
where the perturbation $W_N$ is a solution to 
\begin{align}\label{eq:rem}
    \begin{split}
    -\Delta W_N
    &
    +\mathbb{P}\div
    \left( 
    U_N^{(1)} \otimes U_N^{(2)} 
    +
    U_N^{(2)} \otimes U_N^{(1)}
    +
    U_N^{(2)} \otimes U_N^{(2)}\right.\\
    &
    \left.
    +
    U_N^{(1)} \otimes W_N 
    +
    U_N^{(2)} \otimes W_N
    +
    W_N \otimes U_N^{(1)} 
    +
    W_N \otimes U_N^{(2)}
    +
    W_N \otimes W_N
    \right)=0.
    \end{split}
\end{align}
However, 
in the whole plane case $\mathbb{R}^2$, 
it seems hard to find a function space $X \subset \mathscr{S}'(\mathbb{R}^2)$ in which the following nonlinear estimate holds:
\begin{align}\label{nonlin}
    \left\| (-\Delta)^{-1}\mathbb{P}\div(U \otimes V) \right\|_X
    \leqslant
    C
    \| U \|_X
    \| V \|_X.
\end{align}
In particular, the author \cite{Fujii-pre} implied that \eqref{nonlin} fails for all scaling critical Besov spaces $X=\dB_{p,q}^{\frac{2}{p}-1}(\mathbb{R}^2)$ ($1 \leqslant p,q \leqslant \infty$).
Thus, it seems difficult to construct a function $W_N$ obeying \eqref{eq:rem} and establish its suitable estimate.
Consequently it is hard to prove the desired ill-posedness by the standard argument.

Let us mention the idea to overcome the aforementioned difficulties and prove Theorem \ref{thm:main}.
Inspired by the general observation that the stationary solutions should be the large time behavior of {\it nonstationary} solutions, 
we consider the {\it nonstationary} Navier--Stokes equations.
Then, in contrast to the stationary problem, which possesses difficulties in the singularity of $(-\Delta)^{-1}$ at the origin in the frequency side,
we see that, for the nonstationary Navier--Stokes equations, the heat kernel $\{ e^{t\Delta} \}_{t>0}$ relaxes the singularity on the low-frequency part, and we may obtain the nonlinear estimate
\begin{align}
    \left\| \int_0^t e^{(t-\tau)\Delta}\mathbb{P}\div( u(\tau) \otimes v(\tau) )d\tau \right\|_{X}
    \leqslant
    C
    \| u \|_{X}
    \| v \|_{X}
\end{align}
with $X=\widetilde{L^r}(0,T; \dB_{p,q}^{\frac{2}{p}-1+\frac{2}{r}}(\mathbb{R}^2))$ for some $p,q,r$ and all $0<T \leqslant \infty$. 
See Lemma \ref{lemm:duhamel-est-1} below for details.
Motivated by these facts, we suppose to contrary that \eqref{eq:sNS-1} is well-posed and consider the {\it nonstationary} Navier--Stokes equations with the stationary external forces.
Then, we may show that a contradiction appears from the behavior of the {\it nonstationary} solutions in large times.

Based on the above considerations, we provide the outline of the proof of Theorem \ref{thm:main}. 
Let $\{ F_N \}_{N\in \mathbb{N}}$ be  the external forces defined by \eqref{rem:force}; then it holds that
\begin{align}\label{introFN0}
    \lim_{N\to \infty} \| F_N \|_{\dB_{p,1}^{\frac{2}{p}-3}} = 0.
\end{align}
We consider the nonstationary flow obeying
\begin{align}\label{eq:nNS-4}
    \begin{dcases}
    \partial_t u - \Delta u + \mathbb{P}\div (u \otimes u) =\mathbb{P}F_N, \qquad & t>0, x \in \mathbb{R}^2,\\
    \div u = 0, \qquad & t \geqslant 0, x \in \mathbb{R}^2,\\
    u(0,x)=0, \qquad & x \in \mathbb{R}^2.
    \end{dcases}
\end{align}
By Theorem \ref{thm:ill} below, we may prove the {\it global ill-posedness} of \eqref{eq:nNS-4}; namely there exists a sequence $\{ u_N \}_{N\in \mathbb{N}}$ of solutions to \eqref{eq:nNS-4} on some long time interval $[0,T_N]$ with $T_N \to \infty$ as $N\to \infty$ satisfying 
\begin{align}\label{outline-2}
    \liminf_{N\to \infty} \| u_N(T_N) \|_{\dB_{p,1}^{\frac{2}{p}-1}} \geqslant c
\end{align}
for some positive constant $c$.  
This phenomenon is inherent to two-dimensional flows (see Remark \ref{rem:2D} for details).
Here, we suppose to contrary that \eqref{eq:sNS-1} is well-posed.
Then, we see by \eqref{introFN0} that for sufficiently large $N$, $F_N$ generates a solution $U_N$ to \eqref{eq:sNS-1} satisfying 
\begin{align}\label{U0}
    \lim_{N\to \infty}\| U_N \|_{\dB_{p,1}^{\frac{2}{p}-1}} = 0.
\end{align}
Theorem \ref{thm:stability} below shows that the perturbed equation \eqref{eq:nNS-2} below for $v:=u-U_N$ is globally-in-time solvable, and we obtain a solution $\widetilde{u}_N$ to \eqref{eq:nNS-4} satisfying 
\begin{align}\label{outline-1}
    \sup_{t>0}\| \widetilde{u}_N(t) \|_{\dB_{p,1}^{\frac{2}{p}-1}} \leqslant C \| U_N \|_{\dB_{p,1}^{\frac{2}{p}-1}},
\end{align}
where $C$ is a positive constant independent of $N$.
Then, since the standard uniqueness argument implies that $\widetilde{u}_N(t)=u_N(t)$ holds for all $0 \leqslant t \leqslant T_N$, we see by \eqref{outline-2} and \eqref{outline-1} that
\begin{align}
    0<
    \frac{c}{2}
    \leqslant 
    \| u_N(T_N) \|_{\dB_{p,1}^{\frac{2}{p}-1}}
    =
    \| \widetilde{u}_N (T_N) \|_{\dB_{p,1}^{\frac{2}{p}-1}} 
    \leqslant C \| U_N  \|_{\dB_{p,1}^{\frac{2}{p}-1}}.
\end{align}
for sufficiently large $N$.
Then, letting $N\to \infty$ in the above estimate, we meet a contradiction to \eqref{U0}, which completes the outline of the proof.

This paper is organized as follows.
In Section \ref{sec:pre}, we state the definitions of several function spaces used in this paper and prepare certain key estimates for our analysis.
In Section \ref{sec:nNS}, we focus on the nonstationary Navier--Stokes equations with given stationary external forces and prove that the nonstationary problem is globally ill-posed.
We also show its the global well-posedness under the assumption that the corresponding stationary solution exists if we assume that stationary solutions exist.
Using the results obtained in Section \ref{sec:nNS}, we prove Theorem \ref{thm:main} in Section \ref{sec:pf}.

Throughout this paper, we denote by $C$ and $c$ the constants, which may differ in each line. In particular, $C=C(*,...,*)$ denotes the constant which depends only on the quantities appearing in parentheses. 
Furthermore, we use lowercase for functions with the time and space variables and uppercase for functions that do not depend on the time variable but only on the space variables.

\section{Preliminaries}\label{sec:pre}
In this section, we introduce several function spaces and prepare lemmas, which are to be used in this paper.
Let $\mathscr{S}(\mathbb{R}^2)$ be the set of all Schwartz functions on $\mathbb{R}^2$ and $\mathscr{S}'(\mathbb{R}^2)$ represents the 
set of all tempered distributions on $\mathbb{R}^2$.
We use $L^p(\mathbb{R}^2)$ ($1 \leqslant p \leqslant \infty$) to denote the standard Lebesgue spaces on $\mathbb{R}^2$.
For $F \in \mathscr{S}(\mathbb{R}^2)$, the Fourier transform and inverse Fourier transform of $F$ are defined as
\begin{align}
    \mathscr{F}[F](\xi)=\widehat{F}(\xi):=\int_{\mathbb{R}^2} e^{-ix\cdot\xi} F(x)dx,\qquad
    \mathscr{F}^{-1}[F](x):=\frac{1}{(2\pi)^2}\int_{\mathbb{R}^2} e^{ix\cdot\xi} F(\xi)d\xi.
\end{align}

Let $\{\Phi_j\}_{j \in \mathbb{Z}} \subset \mathscr{S}(\mathbb{R}^2)$ be a dyadic partition of unity satisfying 
\begin{align}
    0 \leqslant \widehat{\Phi_0}(\xi) \leqslant 1,\qquad
    \supp \widehat{\Phi_0} \subset \{ \xi \in \mathbb{R}^2\ ;\ 2^{-1} \leqslant |\xi| \leqslant 2 \},\qquad
    \widehat{\Phi_j}(\xi) = \widehat{\Phi_0}(2^{-j}\xi)
\end{align}
and
\begin{align}
    \sum_{j \in \mathbb{Z}}
    \widehat{\Phi_j}(\xi)
    =1,
    \qquad
    \xi \in \mathbb{R}^2 \setminus \{ 0 \}.
\end{align}
Using this partition of unity, 
we define the Littlewood-Paley dyadic frequency localized operators $\{ \Delta_j \}_{j \in \mathbb{Z}}$ by
$\Delta_j F := \mathscr{F}^{-1}\left[\widehat{\Phi_j} \widehat{F}\right]$
for $j \in \mathbb{Z}$ and $F \in \mathscr{S}'(\mathbb{R}^2)$.
We define the homogeneous Besov spaces $\dB_{p,q}^s(\mathbb{R}^2)$ ($1 \leqslant p,q \leqslant \infty$, $s \in \mathbb{R}$) by 
\begin{align}
    \dB_{p,q}^s(\mathbb{R}^2)
    :={}&
    \left\{
    F \in \mathscr{S}'(\mathbb{R}^2) / \mathscr{P}(\mathbb{R}^2)
    \ ; \ 
    \| F \|_{\dB_{p,q}^s}
    <
    \infty
    \right\},\\
    \| F \|_{\dB_{p,q}^s}
    :={}&
    \left\|
    \left\{
    2^{sj}
    \| \Delta_j F \|_{L^p}
    \right\}_{j \in \mathbb{Z}}
    \right\|_{\ell^{q}},
\end{align}
where $\mathscr{P}(\mathbb{R}^2)$ denotes the set of all polynomials on $\mathbb{R}^2$.
It is well-known that if $s <2/p$ or $(s,q)=(2/p,1)$, then $\dB_{p,q}^s(\mathbb{R}^2)$ is identified as 
\begin{align}\label{chara-Besov}
    \dB_{p,q}^s(\mathbb{R}^2)
    \sim 
    \left\{
    F \in \mathscr{S}'(\mathbb{R}^2)\ ;\ 
    F = \sum_{j \in \mathbb{Z}} \Delta_j F
    \quad {\rm in\ }\mathscr{S}'(\mathbb{R}^2),\quad
    \| F \|_{\dB_{p,q}^s}
    <
    \infty
    \right\}.
\end{align}
See \cite{Saw-18}*{Theorem 2.31} for the proof of \eqref{chara-Besov}.
We refer to \cite{Saw-18} for the basic properties of Besov spaces.

To deal with space-time functions, we use the Chemin--Lerner spaces 
$\widetilde{L^r}(I ; \dB_{p,q}^s(\mathbb{R}^2))$ defined by
\begin{align}
    \widetilde{L^r}(I ; \dB_{p,q}^s(\mathbb{R}^2))
    :={}&
    \left\{
    f : I \to  \mathscr{S}'(\mathbb{R}^2) / \mathscr{P}(\mathbb{R}^2)
    \ ; \ 
    \| f \|_{\widetilde{L^r}(I;\dB_{p,q}^s)}
    <
    \infty
    \right\},\\
    \| f \|_{\widetilde{L^r}(I;\dB_{p,q}^s)}
    :={}&
    \left\|
    \left\{
    2^{sj}
    \| \Delta_j f \|_{L^r(I;L^p)}
    \right\}_{j \in \mathbb{Z}}
    \right\|_{\ell^{q}}
\end{align}
for all  $1 \leqslant p,q,r \leqslant \infty$, $s \in \mathbb{R}$, and intervals $I \subset \mathbb{R}$.
We also use the following notation
\begin{align}
    \widetilde{C}(I ; \dB_{p,q}^s(\mathbb{R}^2))
    :=
    C(I ; \dB_{p,q}^s(\mathbb{R}^2))
    \cap
    \widetilde{L^{\infty}}(I ; \dB_{p,q}^s(\mathbb{R}^2)).
\end{align}
The Chemin--Lerner spaces were first introduced by \cite{Che-Ler-95} and continue to be frequently used for the analysis of compressible viscous fluids in critical Besov spaces.
The Chemin--Lerner spaces possess similar embedding properties as that for usual Besov spaces:
\begin{itemize}
    \item []
        $\widetilde{L^r}(I; \dB_{p,q_1}^s(\mathbb{R}^2)) 
        \hookrightarrow
        \widetilde{L^r}(I; \dB_{p,q_2}^s(\mathbb{R}^2))$ 
        for $1 \leqslant q_1 \leqslant q_2 \leqslant \infty$,
    \item []
        $\widetilde{L^r}(I; \dB_{p_1,q}^{s+\frac{2}{p_1}}(\mathbb{R}^2)) 
        \hookrightarrow
        \widetilde{L^r}(I; \dB_{p_2,q}^{s+\frac{2}{p_2}}(\mathbb{R}^2))$ 
        for $1 \leqslant p_1 \leqslant p_2 \leqslant \infty$.
\end{itemize}
It also holds by the Hausdorff--Young inequality that
\begin{align}
    &
    \widetilde{L^r}(I; \dB_{p,q}^s(\mathbb{R}^2))
    \hookrightarrow
    {L^r}(I; \dB_{p,q}^s(\mathbb{R}^2))
    \ {\rm for\ }1 \leqslant q \leqslant r \leqslant \infty,\\
    &
    {L^r}(I; \dB_{p,q}^s(\mathbb{R}^2)) 
    \hookrightarrow
    \widetilde{L^r}(I; \dB_{p,q}^s(\mathbb{R}^2))
    \ {\rm for\ }1 \leqslant r \leqslant q \leqslant \infty.
\end{align}
See \cite{Bah-Che-Dan-11} for more precise information of the Chemin--Lerner spaces.
One advantage of using the Chemin--Lerner spaces is that there holds the following maximal regularity estimates for the heat kernel $e^{t\Delta}:=G_t*$, where $G_t(x):=(4\pi t)^{-1}e^{-\frac{|x|^2}{4t}}$ ($t>0$, $x \in \mathbb{R}^2$) is the two-dimensional Gaussian.
\begin{lemm}\label{lemm:max-reg}
There exists an absolute positive constant $C$ such that for any $0<T\leqslant \infty$, $1 \leqslant p,q \leqslant \infty$, $1 \leqslant r \leqslant r_0 \leqslant \infty$, and $s \in \mathbb{R}$, it holds
\begin{align}
    \left\| e^{t\Delta} F \right\|_{\widetilde{L^r}(0,T;\dB_{p,q}^{s+\frac{2}{r}})}
    &
    \leqslant
    C
    \| F \|_{\dB_{p,q}^s},\label{max-reg-hom}\\
    \left\| \int_0^t e^{(t-\tau)\Delta} f(\tau) d\tau \right\|_{\widetilde{L^{r_0}}(0,T;\dB_{p,q}^{s+\frac{2}{r_0}})}
    &
    \leqslant
    C
    \| f \|_{\widetilde{L^{r}}(0,T;\dB_{p,q}^{s-2+\frac{2}{r}})}\label{max-reg-inhom}
\end{align}
for all $F \in \dB_{p,q}^s(\mathbb{R}^2)$ and $f \in \widetilde{L^{r}}(0,T;\dB_{p,q}^{s-2+\frac{2}{r}}(\mathbb{R}^2))$.
\end{lemm}

\begin{proof}
It follows from \cite{Bah-Che-Dan-11}*{Corollary 2.5} that there exists an absolute positive constant $C$ such that
\begin{align}
    2^{\frac{2}{r}j}
    \left\| \Delta_j e^{t\Delta}F \right\|_{L^r(0,T;L^p)} 
    &\leqslant 
    C
    \| \Delta_j F \|_{L^p}\\
    2^{\frac{2}{r_0}j}
    \left\| \Delta_j \int_0^t e^{(t-\tau)\Delta} f(\tau) d\tau \right\|_{L^{r_0}(0,T;L^p)}
    &\leqslant
    C
    2^{(-2+\frac{2}{r})j}
    \| \Delta_j f \|_{L^r(0,T;L^p)}
\end{align}
for all $j \in \mathbb{Z}$.
Multiplying these estimates by $2^{sj}$ and taking $\ell^{q}(\mathbb{Z})$-norm, we complete the proof.
\end{proof}

Making use of Lemma \ref{lemm:max-reg}, we derive the following nonlinear estimates.
\begin{lemm}\label{lemm:duhamel-est-2}
Let $0 < T \leqslant \infty$.
Let $p$, $q$, $\sigma$, $\zeta$, $q_1$, $q_2$, $q_3$, $q_4$, $r$, $r_0$, $r_1$, and $r_2$ satisfy
\begin{gather}
    1 \leqslant p, q, \sigma, \zeta , q_1, q_2 \leqslant \infty, \qquad
    1 \leqslant q_3,q_4 \leqslant q, \\
    1 \leqslant r \leqslant r_0, r_1, r_2 \leqslant \infty, \qquad
    2 < r_3, r_4 \leqslant \infty,\\
    1 + \frac{1}{q} = \frac{1}{\sigma} + \frac{1}{\zeta},\qquad
    \frac{1}{\zeta} \leqslant 
    \frac{1}{q_1} + \frac{1}{q_2} , \\
    \max\left\{ 0, 1-\frac{2}{p} \right\} <\frac{1}{r} = \frac{1}{r_1}+\frac{1}{r_2},\qquad
    \frac{1}{r_0} \leqslant \frac{1}{r_3} + \frac{1}{r_4}
\end{gather}
and 
\begin{align}
    2 \leqslant r_3 \leqslant \infty \qquad &{\rm if\ }q_3=1,\\
    2 \leqslant r_4 \leqslant \infty \qquad &{\rm if\ }q_4=1,\\
    \max\left\{ 0, 1-\frac{2}{p} \right\}  \leqslant \frac{1}{r} = \frac{1}{r_1}+\frac{1}{r_2} \qquad &{\rm if\ }q=\sigma=\infty,\ \zeta=1.
\end{align}
Then, there exists an absolute positive constant $C$, independent of all parameters, such that
\begin{align}\label{duhamel-est-2-1}
    \begin{split}
    &\left\|
    \int_0^t e^{(t-\tau)\Delta}(f(\tau)g(\tau)) d\tau 
    \right\|_{\widetilde{L^{r_0}}(0,T ; \dB_{p,q}^{\frac{2}{p}+\frac{2}{r_0}})}\\
    &\quad 
    \leqslant
    C\left( \frac{1}{r} - \max \left\{ 0,1-\frac{2}{p} \right\} \right)^{-\frac{1}{\sigma}}
    \| f \|_{\widetilde{L^{r_1}}(0,T ; \dB_{p,q_1}^{\frac{2}{p}-1+\frac{2}{r_1}})}
    \| g \|_{\widetilde{L^{r_2}}(0,T ; \dB_{p,q_2}^{\frac{2}{p}-1+\frac{2}{r_2}})}
    \\
    &\qquad 
    +
    C\left\{ \left( 1- \frac{2}{r_3} \right)^{-\frac{1}{q_3'}} + \left( 1- \frac{2}{r_4} \right)^{-\frac{1}{q_4'}} \right\}
    \| f \|_{\widetilde{L^{r_3}}(0,T ; \dB_{p,q_3}^{\frac{2}{p}-1+\frac{2}{r_3}})}
    \| g \|_{\widetilde{L^{r_4}}(0,T ; \dB_{p,q_4}^{\frac{2}{p}-1+\frac{2}{r_4}})}
    \end{split}
\end{align}
for all 
\begin{align}
    &
    f \in \widetilde{L^{r_1}}(0,T ; \dB_{p,q_1}^{\frac{2}{p}-1+\frac{2}{r_1}}(\mathbb{R}^2)) \cap  \widetilde{L^{r_3}}(0,T ; \dB_{p,q_3}^{\frac{2}{p}-1+\frac{2}{r_3}}(\mathbb{R}^2)),\\
    &
    g \in \widetilde{L^{r_2}}(0,T ; \dB_{p,q_2}^{\frac{2}{p}-1+\frac{2}{r_2}}(\mathbb{R}^2)) \cap  \widetilde{L^{r_4}}(0,T ; \dB_{p,q_4}^{\frac{2}{p}-1+\frac{2}{r_4}}(\mathbb{R}^2)).
\end{align}
Here, $q_3'$ and $q_4'$ denote the H\"older conjugate exponents of $q_3$ and $q_4$, respectively.
\end{lemm}
\begin{proof}
We first recall the para-product decomposition: 
\begin{align}
    fg
    ={}&
    I_1[f,g] + I_2[f,g] + I_3[f,g],
\end{align}
where
\begin{align}
    &I_1[f,g]
    :=
    \sum_{j \in \mathbb{Z}}\sum_{|i-j|\leqslant 2} \Delta_{i}f \Delta_jg,\\
    &I_2[f,g]
    :=
    \sum_{j \in \mathbb{Z}}\sum_{i \leqslant j-3}\Delta_{i}f \Delta_jg,\\
    &I_3[f,g]
    :=
    \sum_{j \in \mathbb{Z}}\Delta_{j}f \sum_{i \leqslant j-3} \Delta_ig
    =
    I_2[g,f].
\end{align}
We then decompose the left-hand side of \eqref{duhamel-est-2-1} as
\begin{align}\label{J0}
    \begin{split}
    \int_0^t e^{(t-\tau)\Delta}(f(\tau)g(\tau)) d\tau 
    &=
    \sum_{m=1}^3
    \int_0^t e^{(t-\tau)\Delta} I_m[f,g](\tau) d\tau\\
    &
    =:
    \sum_{m=1}^3 J_m[f,g](t).
    \end{split}
\end{align}
We first focus on the estimate for $J_1[f,g]$.
For the case of $1 \leqslant p \leqslant 2$,
Lemma \ref{lemm:max-reg} yields
\begin{align}\label{J-1-1}
    \| J_1[f,g] \|_{\widetilde{L^{r_0}}(0,T ; \dB_{p,q}^{\frac{2}{p}+\frac{2}{r_0}})}
    \leqslant{}
    C\| I_1[f,g] \|_{\widetilde{L^r}(0,T;\dB_{p,q}^{\frac{2}{p}-2+\frac{2}{r}})}
    \leqslant{}
    C\| I_1[f,g] \|_{\widetilde{L^r}(0,T;\dB_{1,q}^{\frac{2}{r}})}.
\end{align}
Using 
\begin{align}\label{DkI1}
    \Delta_kI_1[f,g]
    =
    \Delta_k
    \sum_{|\ell|\leqslant 2}
    \sum_{j \geqslant k-4 } \Delta_{j+\ell}f\Delta_{j}g
\end{align}
and the Hausdorff--Young inequality with $1+1/q=1/\sigma + 1/\zeta$, we have
\begin{align}\label{J-1-2}
    \begin{split}
    &
    \| I_1[f,g] \|_{\widetilde{L^r}(0,T;\dB_{1,q}^{\frac{2}{r}})}\\
    &\quad 
    \leqslant{}
    C\sum_{|\ell|\leqslant 2}\left\{ \sum_{k \in \mathbb{Z}} \left( \sum_{j \geqslant k-4} 2^{\frac{2}{r}(k-j)} 2^{\frac{2}{r}j}\|\Delta_{j+\ell}f\|_{L^{r_1}(0,T;L^p)}\| \Delta_jg  \|_{L^{r_2}(0,T;L^{p'})} \right)^q  \right\}^{\frac{1}{q}}\\
    &\quad 
    \leqslant{}
    C
    \left( \sum_{k \leqslant 4}  2^{\frac{2\sigma}{r}j} \right)^{\frac{1}{\sigma}}
    \sum_{|\ell|\leqslant 2}
    \left\{\sum_{j \in \mathbb{Z}} 
    \left(2^{(\frac{2}{r} + 2(\frac{2}{p}-1))j}\|\Delta_{j+\ell}f\|_{L^{r_1}(0,T;L^p)}\| \Delta_jg  \|_{L^{r_2}(0,T;L^{p})}
    \right)^{\zeta}
    \right\}^{\frac{1}{\zeta}}\\
    &\quad
    \leqslant{}
    Cr^{\frac{1}{\sigma}}
    \| f \|_{\widetilde{L^{r_1}}(0,T ; \dB_{p,q_1}^{\frac{2}{p}-1+\frac{2}{r_1}})}
    \| g \|_{\widetilde{L^{r_2}}(0,T ; \dB_{p,q_2}^{\frac{2}{p}-1+\frac{2}{r_2}})}.
    \end{split}
\end{align}
Here, $p'$ denotes the H\"older conjugate exponent of $p$.
For the case of $p \geqslant 2$, 
we see that
\begin{align}\label{J-1-3}
    \| J_1[f,g] \|_{\widetilde{L^{r_0}}(0,T ; \dB_{p,q}^{\frac{2}{p}+\frac{2}{r_0}})}
    \leqslant{}
    C\| I_1[f,g] \|_{\widetilde{L^r}(0,T;\dB_{p,q}^{\frac{2}{p}-2+\frac{2}{r}})}
    \leqslant{}
    C\| I_1[f,g] \|_{\widetilde{L^r}(0,T;\dB_{\frac{p}{2},q}^{2\mu})},
\end{align}
where we have set $\mu:=1/r-(1-2/p)$.
Using \eqref{DkI1} and the Hausdorff--Young inequality with $1+1/q=1/\sigma + 1/\zeta$, we have
\begin{align}\label{J-1-4}
    \begin{split}
    &
    \| I_1[f,g] \|_{\widetilde{L^r}(0,T;\dB_{\frac{p}{2},q}^{2\mu})}\\
    &\quad 
    \leqslant
    C\sum_{|\ell|\leqslant 2}
    \left\{ \sum_{k \in \mathbb{Z}} \left( \sum_{j \geqslant k-4} 2^{ 2 \mu (k-j) } 2^{ 2 \mu j}\|\Delta_{j+\ell}f\|_{L^{r_1}(0,T;L^p)}\| \Delta_jg  \|_{L^{r_2}(0,T;L^p)} \right)^q  \right\}^{\frac{1}{q}}\\
    &\quad 
    \leqslant{}
    C
    \left( \sum_{k \leqslant 4}  2^{2 \mu \sigma j} \right)^{\frac{1}{\sigma}}\\
    &\qquad\times
    \sum_{|\ell|\leqslant 2}
    \left\{
    \sum_{j \in \mathbb{Z}} 
    \left(
    2^{(\frac{2}{p}-1+\frac{2}{r_1})j}\|\Delta_{j+\ell}f\|_{L^{r_1}(0,T;L^p)}2^{(\frac{2}{p}-1+\frac{2}{r_2})j}\| \Delta_jg  \|_{L^{r_2}(0,T;L^{p})}
    \right)^{\zeta}
    \right\}^{\frac{1}{\zeta}}\\
    &\quad
    \leqslant{}
    C\left( \frac{1}{r} - \left( 1- \frac{2}{p} \right) \right)^{-\frac{1}{\sigma}}
    \| f \|_{\widetilde{L^{r_1}}(0,T ; \dB_{p,q_1}^{\frac{2}{p}-1+\frac{2}{r_1}})}
    \| g \|_{\widetilde{L^{r_2}}(0,T ; \dB_{p,q_2}^{\frac{2}{p}-1+\frac{2}{r_2}})}.
    \end{split}
\end{align}
Thus, combining the estimates \eqref{J-1-1}, \eqref{J-1-2}, \eqref{J-1-3}, and \eqref{J-1-4}, we obtain 
\begin{align}\label{J1}
    \begin{split}
    &\| J_1[f,g] \|_{\widetilde{L^{r_0}}(0,T ; \dB_{p,q}^{\frac{2}{p}+\frac{2}{r_0}})}\\
    &\quad\leqslant{}
    C\left( \frac{1}{r} - \max \left\{ 0,1-\frac{2}{p} \right\} \right)^{-\frac{1}{\sigma}}
    \| f \|_{\widetilde{L^{r_1}}(0,T ; \dB_{p,q_1}^{\frac{2}{p}-1+\frac{2}{r_1}})}
    \| g \|_{\widetilde{L^{r_2}}(0,T ; \dB_{p,q_2}^{\frac{2}{p}-1+\frac{2}{r_2}})}
    \end{split}
\end{align}
for all $1 \leqslant p \leqslant \infty$.

Next, we consider the estimate for $J_2[f,g]$ and $J_3[f,g]$.
Let $1 \leqslant \rho \leqslant \infty$ satisfy $1/\rho = 1/r_3 + 1/r_4$.
It follows from Lemma \ref{lemm:max-reg} that
\begin{align}
    \| J_2[f,g] \|_{\widetilde{L^{r_0}}(0,T ; \dB_{p,q}^{\frac{2}{p}+\frac{2}{r_0}})}
    \leqslant{}
    C\| I_2[f,g] \|_{\widetilde{L^{\rho}}(0,T;\dB_{p,q}^{\frac{2}{p}-2+\frac{2}{\rho}})}.
\end{align}
Using
\begin{align}
    \Delta_kI_2[f,g]
    =
    \Delta_k\sum_{|\ell|\leqslant 3}\sum_{j \leqslant k+\ell-3}\Delta_{j}f\Delta_{k+\ell}g,
\end{align}
we see that 
\begin{align}
    &
    \|\Delta_kI_2[f,g]\|_{L^{\rho}(0,T;L^p)}\\
    &\quad\leqslant{}
    C
    \sum_{|\ell|\leqslant 3}
    \sum_{j \leqslant k+\ell-3}
    \| \Delta_j f \|_{L^{r_3}(0,T;L^{\infty})}
    \| \Delta_{k+\ell}g\|_{L^{r_4}(0,T;L^p)}\\
    &\quad\leqslant{}
    C
    \left(
    \sum_{j \leqslant k}
    2^{(1-\frac{2}{r_3})q_3'j}
    \right)^{\frac{1}{q_3'}}
    \| f \|_{\widetilde{L^{r_3}}(0,T ; \dB_{\infty,q_3}^{-1+\frac{2}{r_3}})}
    \sum_{|\ell|\leqslant 3}
    \| \Delta_{k+\ell}g\|_{L^{r_4}(0,T;L^p)}
    \\
    &\quad\leqslant{}
    C
    2^{(1-\frac{2}{r_3})k}
    \left(
    1-\frac{2}{r_3}
    \right)^{-\frac{1}{q_3'}}
    \| f \|_{\widetilde{L^{r_3}}(0,T ; \dB_{p,q_3}^{\frac{2}{p}-1+\frac{2}{r_3}})}
    \sum_{|\ell|\leqslant 3}
    \| \Delta_{k+\ell}g \|_{L^{r_4}(0,T;L^p)}.
\end{align}
Multiplying this by $2^{(\frac{2}{p}-1+\frac{2}{\rho})k}$ and taking $\ell^{q}(\mathbb{Z})$ norm with respect to $k$, we obtain 
\begin{align}\label{J2}
    \begin{split}
    &\| J_2[f,g] \|_{\widetilde{L^{r_0}}(0,T ; \dB_{p,q}^{\frac{2}{p}+\frac{2}{r_0}})}\\
    &\quad\leqslant
    C
    \left( 1- \frac{2}{r_3} \right)^{-\frac{1}{q_3'}}
    \| f \|_{\widetilde{L^{r_3}}(0,T ; \dB_{p,q_3}^{\frac{2}{p}-1+\frac{2}{r_3}})}
    \| g \|_{\widetilde{L^{r_4}}(0,T ; \dB_{p,q}^{\frac{2}{p}-1+\frac{2}{r_4}})}\\
    &\quad
    \leqslant
    C
    \left( 1- \frac{2}{r_3} \right)^{-\frac{1}{q_3'}}
    \| f \|_{\widetilde{L^{r_3}}(0,T ; \dB_{p,q_3}^{\frac{2}{p}-1+\frac{2}{r_3}})}
    \| g \|_{\widetilde{L^{r_4}}(0,T ; \dB_{p,q_4}^{\frac{2}{p}-1+\frac{2}{r_4}})}.
    \end{split}
\end{align}
By the same argument, we also see that
\begin{align}\label{J3}
    \begin{split}
    &\| J_3[f,g] \|_{\widetilde{L^{r_0}}(0,T ; \dB_{p,q}^{\frac{2}{p}+\frac{2}{r_0}})}\\
    &\quad\leqslant
    C
    \left( 1- \frac{2}{r_4} \right)^{-\frac{1}{q_4'}}
    \| g \|_{\widetilde{L^{r_4}}(0,T ; \dB_{p,q_4}^{\frac{2}{p}-1+\frac{2}{r_4}})}
    \| f \|_{\widetilde{L^{r_3}}(0,T ; \dB_{p,q_3}^{\frac{2}{p}-1+\frac{2}{r_3}})}.
    \end{split}
\end{align}
Collecting \eqref{J0}, \eqref{J1}, \eqref{J2}, and \eqref{J3}, we complete the proof.
\end{proof}
Let us apply Lemma \ref{lemm:duhamel-est-2} to obtain several estimates for the nonlinear Duhamel integral defined by
\begin{align}\label{df-duhamel}
    \mathcal{D}[u,v](t)
    :=
    -
    \int_0^t e^{(t-\tau)\Delta} \mathbb{P}\div (u(\tau) \otimes v(\tau)) d\tau
\end{align}
for two space-time vector fields $u=(u_1(t,x),u_2(t,x))$ and $v=(v_1(t,x),v_2(t,x))$ ($t>0$, $x \in \mathbb{R}^2$).
\begin{lemm}\label{lemm:duhamel-est-1}
Let $0 < T \leqslant \infty$.
Let $p$, $q$, $r$, $r_0$, $r_1$, and $r_2$ satisfy
\begin{gather}
    1 \leqslant p,q,r \leqslant \infty,\qquad
    r \leqslant r_0 \leqslant \infty,\qquad
    2 < r_1,r_2 \leqslant \infty\\
    \max \left\{ 0, 1-\frac{2}{p} \right\}
    < \frac{1}{r}
    = \frac{1}{r_1} + \frac{1}{r_2}
\end{gather}
and $2 \leqslant r_1,r_2 \leqslant \infty$ if $q=1$.
Then, there exists a positive constant $C=C(p,q,r,r_0,r_1,r_2)$ such that
\begin{align}\label{duhamel-est-1-1}
    \left\| \mathcal{D}[u,v] \right\|_{\widetilde{L^{r_0}}(0,T;\dB_{p,q}^{\frac{2}{p}-1+\frac{2}{r_0}})}
    \leqslant
    C
    \left\| u \right\|_{\widetilde{L^{r_1}}(0,T;\dB_{p,q}^{\frac{2}{p}-1+\frac{2}{r_1}})}
    \left\| v \right\|_{\widetilde{L^{r_2}}(0,T;\dB_{p,q}^{\frac{2}{p}-1+\frac{2}{r_2}})}
\end{align}
for all 
$u \in \widetilde{L^{r_1}}(0,T;\dB_{p,q}^{\frac{2}{p}-1+\frac{2}{r_1}}(\mathbb{R}^2))$
and
$v \in \widetilde{L^{r_2}}(0,T;\dB_{p,q}^{\frac{2}{p}-1+\frac{2}{r_2}}(\mathbb{R}^2))$.
\end{lemm}
\begin{proof}
Let $\sigma=1$, $\zeta=q_1=q_2=q_3=q_4=q$, $r_3=r_1$, and $r_4=r_2$.
Then, Lemma \ref{lemm:duhamel-est-2} yields 
\begin{align}
    \left\| \mathcal{D}[u,v] \right\|_{\widetilde{L^{r_0}}(0,T;\dB_{p,q}^{\frac{2}{p}-1+\frac{2}{r_0}})}
    \leqslant{}&
    C
    \sum_{k,\ell=1}^2
    \left\|
    \int_0^t e^{(t-\tau)\Delta}(u_k(\tau)v_{\ell}(\tau)) d\tau
    \right\|_{\widetilde{L^{r_0}}(0,T;\dB_{p,q}^{\frac{2}{p}+\frac{2}{r_0}})}\\
    \leqslant{}&
    C
    \left\| u \right\|_{\widetilde{L^{r_1}}(0,T;\dB_{p,q}^{\frac{2}{p}-1+\frac{2}{r_1}})}
    \left\| v \right\|_{\widetilde{L^{r_2}}(0,T;\dB_{p,q}^{\frac{2}{p}-1+\frac{2}{r_2}})}
\end{align}
and this completes the proof.
\end{proof}
\begin{lemm}\label{lemm:prod-2}
    Let $0<T<\infty$ and $1 \leqslant p \leqslant 2$.
    Then there exists a positive constant $K_0=K_0(p)$ such that 
    \begin{align}\label{duhamel-est-3}
        \sup_{0 \leqslant t \leqslant T}
        \| \mathcal{D}[u,v](t) \|_{\dB_{p,\infty}^{\frac{2}{p}-1}}
        \leqslant
        K_0
        \| u \|_{\widetilde{L^{\infty}}(0,T;\dB_{p,1}^{\frac{2}{p}-1})}
        \sup_{0 \leqslant t \leqslant T}
        \| v(t) \|_{\dB_{p,\infty}^{\frac{2}{p}-1}}
    \end{align}
    for all $u \in \widetilde{C}([0,T];\dB_{p,1}^{\frac{2}{p}-1}(\mathbb{R}^2))$ and $v \in C([0,T];\dB_{p,\infty}^{\frac{2}{p}-1}(\mathbb{R}^2))$.
\end{lemm}
\begin{proof}
We set $\zeta=q_1=q_3=1$ and $q=q=q_2=q_4=\sigma=r=r_0=r_1=r_2=\infty$.
Then, from Lemma \ref{lemm:duhamel-est-2}, we have
\begin{align}
        \sup_{0 \leqslant t \leqslant T}
        \| \mathcal{D}[u,v](t) \|_{\dB_{p,\infty}^{\frac{2}{p}-1}}
        \leqslant{}&
        C
        \sum_{k,\ell=1}^2
        \left\|
        \int_0^t e^{(t-\tau)\Delta}(u_k(\tau)v_{\ell}(\tau)) d\tau
        \right\|_{\widetilde{L^{\infty}}(0,T;\dot{B}_{p,\infty}^{\frac{2}{p}})}\\
        \leqslant{}&
        C
        \| u \|_{\widetilde{L^{\infty}}(0,T;\dB_{p,1}^{\frac{2}{p}-1})}
        \sup_{0 \leqslant t \leqslant T}
        \| v(t) \|_{\dB_{p,\infty}^{\frac{2}{p}-1}}
    \end{align}
and complete the proof.
\end{proof}
Finally, we state a couple of two estimates, which plays a key role in the proof of Theorem \ref{thm:ill} below.
\begin{lemm}\label{lemm:ill-duhamel-est}
There exists an absolute positive constant $C$ such that
for any $0 < T \leqslant \infty$, $1 \leqslant p \leqslant 2$, and $3 \leqslant N < \infty$, it holds
\begin{align}
    &
    \begin{aligned}\label{ill-duhamel-est-1}
    \left\|
    \mathcal{D}[u,v] 
    \right\|_{\widetilde{L^{\infty}}(0,T;\dot{B}_{p,1}^{\frac{2}{p}-1})}
    \leqslant{}
    &
    CN
    \| u \|_{\widetilde{L^N}(0,T;\dot{B}_{p,2}^{\frac{2}{p}-1+\frac{2}{N}})}
    \| v \|_{\widetilde{L^N}(0,T;\dot{B}_{p,2}^{\frac{2}{p}-1+\frac{2}{N}})}\\
    &
    +
    C
    \| u \|_{\widetilde{L^{\infty}}(0,T;\dot{B}_{p,1}^{\frac{2}{p}-1})}
    \| v \|_{\widetilde{L^{\infty}}(0,T;\dot{B}_{p,1}^{\frac{2}{p}-1})},
    \end{aligned}\\
    &
    \begin{aligned}\label{ill-duhamel-est-2}
    \left\|
    \mathcal{D}[u,v]
    \right\|_{\widetilde{L^N}(0,T;\dot{B}_{p,2}^{\frac{2}{p}-1+\frac{2}{N}})}
    \leqslant
    C\sqrt{N}
    \| u \|_{\widetilde{L^N}(0,T;\dot{B}_{p,2}^{\frac{2}{p}-1+\frac{2}{N}})}
    \| v \|_{\widetilde{L^N}(0,T;\dot{B}_{p,2}^{\frac{2}{p}-1+\frac{2}{N}})}
    \end{aligned}
\end{align}
for all
$u,v \in \widetilde{L^{\infty}}(0,T;\dot{B}_{p,1}^{\frac{2}{p}-1}(\mathbb{R}^2)) \cap \widetilde{L^N}(0,T;\dot{B}_{p,2}^{\frac{2}{p}-1+\frac{2}{N}}(\mathbb{R}^2))$.
\end{lemm}
\begin{proof}
Using Lemma \ref{lemm:duhamel-est-2} with $q=\sigma=\zeta=1$, $q_1=q_2=2$, $q_3=q_4=1$, $r_0=r_3=r_4=\infty$, $r=N/2$, and $r_1=r_2=N$, we obtain 
\begin{align}
    \left\|
    \mathcal{D}[u,v] 
    \right\|_{\widetilde{L^{\infty}}(0,T;\dot{B}_{p,1}^{\frac{2}{p}-1})}
    \leqslant{}&
    C
    \sum_{k,\ell=1}^2
    \left\|
    \int_0^t e^{(t-\tau)\Delta}(u_k(\tau)v_{\ell}(\tau)) d\tau
    \right\|_{\widetilde{L^{\infty}}(0,T;\dot{B}_{p,1}^{\frac{2}{p}})}\\
    \leqslant{}&
    CN
    \| u \|_{\widetilde{L^N}(0,T;\dot{B}_{p,2}^{\frac{2}{p}-1+\frac{2}{N}})}
    \| v \|_{\widetilde{L^N}(0,T;\dot{B}_{p,2}^{\frac{2}{p}-1+\frac{2}{N}})}\\
    &
    +
    C
    \| u \|_{\widetilde{L^{\infty}}(0,T;\dot{B}_{p,1}^{\frac{2}{p}-1})}
    \| v \|_{\widetilde{L^{\infty}}(0,T;\dot{B}_{p,1}^{\frac{2}{p}-1})},
\end{align}
which implies \eqref{ill-duhamel-est-1}.
From Lemma \ref{lemm:duhamel-est-2} with $q=\sigma=q_1=q_2=q_3=q_4=2$, $\zeta=1$, $r=N/2$, and $r_0=r_1=r_2=r_3=r_4=N$, it follows that
\begin{align}
    \left\|
    \mathcal{D}[u,v]
    \right\|_{\widetilde{L^N}(0,T;\dot{B}_{p,2}^{\frac{2}{p}-1+\frac{2}{N}})}
    \leqslant{}&
    C
    \sum_{k,\ell=1}^2
    \left\|
    \int_0^t e^{(t-\tau)\Delta}(u_k(\tau)v_{\ell}(\tau)) d\tau
    \right\|_{\widetilde{L^N}(0,T;\dot{B}_{p,2}^{\frac{2}{p}+\frac{2}{N}})}\\
    \leqslant{}&
    C\sqrt{N}
    \| u \|_{\widetilde{L^N}(0,T;\dot{B}_{p,2}^{\frac{2}{p}-1+\frac{2}{N}})}
    \| v \|_{\widetilde{L^N}(0,T;\dot{B}_{p,2}^{\frac{2}{p}-1+\frac{2}{N}})}.
\end{align}
Thus, we have \eqref{ill-duhamel-est-2} and complete the proof.
\end{proof}

\section{Nonstationary analysis}\label{sec:nNS}
Let us consider the {\it nonstationary} incompressible Navier--Stokes equations with the {\it stationary} external force:
\begin{align}\label{eq:nNS-0}
    \begin{cases}
        \partial_t u - \Delta u + \mathbb{P}\div(u \otimes u) = \mathbb{P}F, \qquad & t>0,x \in \mathbb{R}^2,\\
        \div u = 0, \qquad &  t\geqslant 0,x \in \mathbb{R}^2,\\
        u(0,x)=0, \qquad & x \in \mathbb{R}^2.
    \end{cases}
\end{align}
Here, $u=u(t,x):(0,\infty) \times \mathbb{R}^2 \to \mathbb{R}^2$ denote the unknown {\it nonstationary} velocity of the fluid,
and $F=F(x):\mathbb{R}^2 \to \mathbb{R}^2$ is the given {\it stationary} external force.
By the Duhamel principle and 
\begin{align}
    \int_0^te^{(t-\tau)\Delta}\mathbb{P}Fd\tau=(-\Delta)^{-1}\left(1-e^{t\Delta}\right)\mathbb{P}F,
\end{align}
the equation \eqref{eq:nNS-0} is formally equivalent to
\begin{align}\label{int:nNS-0}
    u(t)=(-\Delta)^{-1}\left(1-e^{t\Delta}\right)\mathbb{P}F + \mathcal{D}[u,u](t),
\end{align}
where the nonlinear Duhamel term $\mathcal{D}[\cdot , \cdot]$ is defined in \eqref{df-duhamel}.
We say that $u$ is a mild solution to \eqref{eq:nNS-0} if $u$ satisfies \eqref{int:nNS-0}.

\subsection{Global ill-posedness}\label{subseq:2}
Since the external force in \eqref{eq:nNS-0} does not depends on time, 
it is excepted that the solution to \eqref{eq:nNS-0} does not decay in time.
However, it is difficult to close the nonlinear estimates in the scaling critical spaces that include functions non-decaying in time such as $\widetilde{L^{\infty}}([0,\infty);\dB_{p,q}^{\frac{2}{p}-1}(\mathbb{R}^2))$ (see Lemmas \ref{lemm:duhamel-est-1} and \ref{lemm:prod-2}).
Thus, it is hard to construct a bounded-in-time global solution to \eqref{eq:nNS-0}.
In this subsection,
we justify the above consideration in the sense that for every $1 \leqslant p \leqslant 2$, the solution map $\dB_{p,1}^{\frac{2}{p}-3}(\mathbb{R}^2) \ni F \mapsto u \in \widetilde{C}( [0,\infty) ; \dB_{p,1}^{\frac{2}{p}-1}(\mathbb{R}^2) )$ is {\it discontinuous} even if it exists.
More precisely we show that there exist 
two sequences $\{ F_N \}_{N \in \mathbb{N}} \subset \dB_{p,1}^{\frac{2}{p}-3}(\mathbb{R}^2)$ of external forces
and $\{ T_N \}_{N\in \mathbb{N}} \subset (0,\infty)$ of times 
satisfying 
\begin{align}
    \lim_{N \to \infty} \| F_N \|_{\dB_{p,1}^{\frac{2}{p}-3}} = 0,\qquad
    \lim_{N \to \infty} T_N=\infty,
\end{align}
such that \eqref{eq:nNS-0} with the external force $F_N$ admits a solution $u_N \in \widetilde{C}([0,T_N];\dB_{p,1}^{\frac{2}{p}-1}(\mathbb{R}^2))$ satisfying
\begin{align}
    \liminf_{N\to \infty}\| u_N(T_N) \|_{\dB_{p,1}^{\frac{2}{p}-1}} > 0. 
\end{align}
In this paper, we call this phenomenon as the {\it global ill-posedness}.
The aim of this subsection is to prove the following theorem.
\begin{thm}\label{thm:ill}
Let $1 \leqslant p \leqslant 2$.
Then, there exist two positive constants $\delta_1=\delta_1(p)$ and $K_1=K_1(p)$ such that for any $0<\delta \leqslant \delta_1$, there exists a sequence $\{ F_{\delta,N} \}_{N \in \mathbb{N}} \subset \dB_{p,1}^{\frac{2}{p}-3}(\mathbb{R}^2)$ of external forces such that the following two statements are true:
\begin{itemize}
    \item [(i)]
    For any $N\in\mathbb{N}$, it holds
    \begin{align}
        \| F_{\delta,N} \|_{\dB_{p,1}^{\frac{2}{p}-3}}\leqslant \frac{K_1\delta}{\sqrt{N}}.
    \end{align}
    \item [(ii)] 
    Let $T_N:=2^{2N}$.
    Then,
    for each integer $N \geqslant 3$, \eqref{eq:nNS-0} with the external force $F_{\delta,N}$ admits a mild solution $u_{\delta,N} \in \widetilde{C}([0,T_N];\dB_{p,1}^{\frac{2}{p}-1}(\mathbb{R}^2))$ 
    satisfying 
    \begin{align}\label{ill-est}
        \liminf_{N\to \infty} \| u_{\delta,N}(T_N) \|_{\dB_{p,1}^{\frac{2}{p}-1}} > \frac{\delta^2}{K_1},\quad
        \limsup_{N\to \infty} \| u_{\delta,N} \|_{\widetilde{L^{\infty}}(0,T_N;\dB_{p,1}^{\frac{2}{p}-1})} < K_1\delta^2.
    \end{align}
\end{itemize}
\end{thm}
\begin{rem}\label{rem:2D}
    For the nonstationary Navier--Stokes equations in $\mathbb{R}^n$ with $n \geqslant 3$,
    it is possible to construct a small global-in-time unique solution for small external force that is bounded-in-time but does not decay as $t \to \infty$. 
    We refer to \cites{Gei-Hie-Ngu-16,Koz-Nak-96} and references therein for the time periodic setting.
    Thus, the assertion of Theorem \ref{thm:ill} is one of phenomena inherent to two-dimensional flows.
\end{rem}
As the proof of Theorem \ref{thm:ill} is the most complicated part of this paper, we shall sketch its outline before starting on the rigorous proof.
We first follow the standard ill-posedness argument used in studies such as \cites{Bou-Pav-08,Yon-10} and formally decompose the solution $u_{\delta,N}$ as 
\begin{align}
    u_{\delta,N} = u_{\delta,N}^{(1)} + u_{\delta,N}^{(2)} + w_{\delta,N},
\end{align}
where $u_{\delta,N}^{(1)}$ and $u_{\delta,N}^{(2)}$ denote the first and second iterations, respectively, which are defined by 
\begin{align}
    u_{\delta,N}^{(1)}(t):= (-\Delta)^{-1}\left(1-e^{t\Delta}\right)\mathbb{P}F_{\delta,N},\qquad
    u_{\delta,N}^{(2)}(t):=\mathcal{D}\left[u_{\delta,N}^{(1)},u_{\delta,N}^{(1)}\right](t)
\end{align}
and $w_{\delta,N}$ is the perturbation obeying \eqref{eq:wN} below.
Then, choosing a suitable sequence $\{ F_{\delta,N} \}_{N \in \mathbb{N}} \subset \dB_{p,1}^{\frac{2}{p}-3}(\mathbb{R}^2)$, we may see that 
\begin{align}\label{ill-explain-1}
    \| F_{\delta,N} \|_{\dB_{p,1}^{\frac{2}{p}-3}}\leqslant C\frac{\delta}{\sqrt{N}},\qquad
    \left\| u_{\delta,N}^{(1)} \right\|_{\widetilde{L^{\infty}}(0,\infty;\dB_{p,1}^{\frac{2}{p}-1})} \leqslant C \frac{\delta}{\sqrt{N}},
\end{align}
whereas the second iteration satisfies
\begin{align}\label{ill-explain-2}
    \left\| u_{\delta,N}^{(2)}(T_N) \right\|_{\dB_{p,1}^{\frac{2}{p}-1}}\geqslant c\delta^2,\qquad
    \left\| u_{\delta,N}^{(2)} \right\|_{\widetilde{L^{\infty}}(0,T_N;\dB_{p,1}^{\frac{2}{p}-1})} \leqslant C\delta^2
\end{align}
for sufficiently large $N$.
It is relatively easy to obtain \eqref{ill-explain-1} and \eqref{ill-explain-2}, 
while the most difficult part of the proof is how to construct and control the perturbation $w_{\delta,N}$.
To this end, we consider the estimate of $w_{\delta,N}$ in 
\begin{align}
    \widetilde{C}( [0,T_N] ; \dB_{p,1}^{\frac{2}{p}-1}(\mathbb{R}^2)) \cap \widetilde{L^N}( 0,T_N ; \dB_{p,2}^{\frac{2}{p}-1+\frac{2}{N}}(\mathbb{R}^2)).
\end{align}
Here, the choice of the auxiliary space $\widetilde{L^N}( 0,T_N ; \dB_{p,2}^{\frac{2}{p}-1+\frac{2}{N}}(\mathbb{R}^2))$ is the most crucial idea of the proof.
Indeed, choosing the Lebesgue exponent of the time integral as $N$, we see that the $L^N(0,T_N)$-norm of functions are bounded by the $L^{\infty}(0,T_N)$-norm with the constant independent of $N$.
More precisely, it holds
\begin{align}
    \| f \|_{L^N(0,T_N)}
    \leqslant
    T_N^\frac{1}{N}
    \| f \|_{L^{\infty}(0,T_N)}
    =
    4
    \| f \|_{L^{\infty}(0,T_N)}
\end{align}
for all $f \in L^{\infty}(0,T_N)$.
On the other hand, choosing the interpolation index as $q=2$ in the auxiliary Chemin--Lerner space $\widetilde{L^N}( 0,T_N ; \dB_{p,2}^{\frac{2}{p}-1+\frac{2}{N}}(\mathbb{R}^2))$, we may use a pair of estimates \eqref{ill-duhamel-est-1} and \eqref{ill-duhamel-est-2} in Lemma \ref{lemm:ill-duhamel-est} above.
Then, keeping these facts in mind and making use of the iterative argument via Lemma \ref{lemm:ill-duhamel-est}, we may obtain the existence of the perturbation $w_{\delta,N}$ and the estimate
\begin{align}\label{ill-explain-3}
    \| w_{\delta,N} \|_{\widetilde{L^{\infty}}(0,T_N;\dB_{p,1}^{\frac{2}{p}-1})}
    \leqslant
    C\delta^3,\qquad
    \| w_{\delta,N} \|_{\widetilde{L^N}(0,T_N;\dB_{p,2}^{\frac{2}{p}-1+\frac{2}{N}})}
    \leqslant
    C\frac{\delta^3}{\sqrt{N}}
\end{align}
for sufficiently small $\delta$.
Collecting \eqref{ill-explain-1}, \eqref{ill-explain-2}, and \eqref{ill-explain-3}, we obtain the solution $u_{\delta,N}$ satisfying the desired estimate \eqref{ill-est}.

Now, the rigorous proof of Theorem \ref{thm:ill} reads as follows.
\begin{proof}[Proof of Theorem \ref{thm:ill}]
We split the proof into five parts. 
In the first step, we provide the definition and an estimate for the sequence of the external forces.
In the second and third steps, we establish some estimates on the first and second iterations, respectively.
In the fourth step, we construct the remaining part of the solution and prepare it's estimates. 
In the final step, we make use of various estimates established in the previous steps and complete the proof. 

\noindent
{\it Step.1 The definition and estimate for the sequence of external forces.}
Let $N \geqslant 3$ be an integer, and let $0< \delta \leqslant 1$.
We choose a function $\Psi \in \mathscr{S}(\mathbb{R}^2)$ satisfying
\begin{align}\label{Psi}
    \begin{dcases}
    \widehat{\Psi}{\rm\ is\ radial\ symmetric,}\\
    0 \leqslant \widehat{\Psi}(\xi) \leqslant 1, \\
    \supp \widehat{\Psi} \subset \{ \xi \in \mathbb{R}^2 \ ;\ |\xi| \leqslant 2 \},\\
    \widehat{\Psi}(\xi) =1 \quad {\rm for\ all\ }\xi \in \mathbb{R}^2{\rm \ with\ }|\xi|\leqslant 1.
    \end{dcases}
\end{align}
We define the external force $F_{\delta,N}$ as
\begin{align}\label{df-FN}
    F_{\delta,N}:=-\Delta \widetilde{F}_{\delta,N},\qquad
    \widetilde{F}_{\delta,N}
    := \frac{\delta}{\sqrt{N}}\nabla^{\perp}\left(\Psi(x)\cos(Mx_1)\right),
\end{align}
where $M \geqslant 10$ is a positive constant to be determined later.
We note that $F_{\delta,N}$ is a real valued function satisfying $\div F_{\delta,N}=0$.
Here, since
\begin{align}
    \mathscr{F}[\Psi(x)\cos(Mx_1)](\xi)
    =
    \frac{\widehat{\Psi}(\xi+Me_1) + \widehat{\Psi}(\xi-Me_1)}{2},
\end{align}
it holds
\begin{align}\label{supp-psi}
    \supp \widehat{\widetilde{F}_{\delta,N}} \subset \{ \xi \in \mathbb{R}^2\ ;\ M-2 \leqslant |\xi| \leqslant M+2 \}.
\end{align}
Thus, we easily see that
\begin{align}\label{est-FN}
    \| F_{\delta,N} \|_{\dB_{p,1}^{\frac{2}{p}-3}}
    \leqslant
    C 
    \left\| \widetilde{F}_{\delta,N} \right\|_{\dB_{p,1}^{\frac{2}{p}-1}}
    \leqslant
    CM^{\frac{2}{p}}
    \frac{\delta}{\sqrt{N}}.
\end{align}

\noindent
{\it Step.2 The estimates for the first iteration.}
Let $u_{\delta,N}^{(1)}$ be the first iteration defined by 
\begin{align}\label{df-u1}
    \begin{split}
    u_{\delta,N}^{(1)}(t)
    :=
    (-\Delta)^{-1}\left(1-e^{t\Delta}\right)\mathbb{P}F_{\delta,N}
    =\left(1-e^{t\Delta}\right)\widetilde{F}_{\delta,N}
    \end{split}
\end{align}
Then, it follows from Lemma \ref{lemm:max-reg}, \eqref{est-FN}, and \eqref{df-u1} that
\begin{align}\label{est-u(1)-1}
    \left\| u_{\delta,N}^{(1)} \right\|_{\widetilde{L^{\infty}}(0,\infty;\dB_{p,1}^{\frac{2}{p}-1})}
    \leqslant
    C\left\|\widetilde{F}_{\delta,N}\right\|_{\dB_{p,1}^{\frac{2}{p}-1}}
    \leqslant
    CM^{\frac{2}{p}}\frac{\delta}{\sqrt{N}}
\end{align}
and
\begin{align}\label{est-u(1)-2}
    \begin{split}
    \left\| u_{\delta,N}^{(1)} \right\|_{\widetilde{L^N}(0,T_N;\dB_{p,2}^{\frac{2}{p}-1+\frac{2}{N}})}
    \leqslant{}&
    T_N^{\frac{1}{N}}
    \left\|\widetilde{F}_{\delta,N}\right\|_{\dB_{p,1}^{\frac{2}{p}-1+\frac{2}{N}}}
    +
    \left\| e^{t\Delta}\widetilde{F}_{\delta,N} \right\|_{\widetilde{L^N}(0,T_N;\dB_{p,2}^{\frac{2}{p}-1+\frac{2}{N}})}\\
    \leqslant{}&
    CM^{\frac{2}{p}+\frac{2}{N}}\frac{\delta}{\sqrt{N}}
    +
    C\left\|\widetilde{F}_{\delta,N}\right\|_{\dB_{p,1}^{\frac{2}{p}-1}}\\
    \leqslant{}&
    CM^{\frac{2}{p}+1}\frac{\delta}{\sqrt{N}}.
    \end{split}
\end{align}
Here, we have used $T_N^{\frac{1}{N}}=4$.


\noindent
{\it Step.3 The estimates for the second iteration.}
Next, we consider the second iteration:
\begin{align}
    u_{\delta,N}^{(2)}(t)
    :=
    \mathcal{D}\left[u_{\delta,N}^{(1)},u_{\delta,N}^{(1)}\right](t)
    =
    -\int_0^t e^{(t-\tau)\Delta}\mathbb{P}\div \left(u_1^{(1)}(\tau) \otimes u_1^{(1)}(\tau)\right) d\tau.
\end{align}
We decompose $u_{\delta,N}^{(2)}$ as 
\begin{align}
    u_{\delta,N}^{(2)}
    ={}
    \mathcal{D}\left[\left(1-e^{\tau \Delta}\right)\widetilde{F}_{\delta,N},\left(1-e^{\tau \Delta}\right)\widetilde{F}_{\delta,N}\right]
    ={}
    u_{\delta,N}^{(2,1)}+u_{\delta,N}^{(2,2)},
\end{align}
where
\begin{align}
    u_{\delta,N}^{(2,1)}
    :={}&
    \mathcal{D}\left[\widetilde{F}_{\delta,N},\widetilde{F}_{\delta,N}\right],\\
    u_{\delta,N}^{(2,2)}
    :={}&
    -
    \mathcal{D}\left[e^{\tau \Delta}\widetilde{F}_{\delta,N},\widetilde{F}_{\delta,N}\right]
    -
    \mathcal{D}\left[\widetilde{F}_{\delta,N},e^{\tau \Delta}\widetilde{F}_{\delta,N}\right]
    +
    \mathcal{D}\left[e^{\tau \Delta}\widetilde{F}_{\delta,N},e^{\tau \Delta}\widetilde{F}_{\delta,N}\right].
\end{align}
We focus on the estimate of $u_{\delta,N}^{(2,1)}$. 
We note that it holds
\begin{align}
    u_{\delta,N}^{(2,1)}
    =
    -
    (-\Delta)^{-1}
    \left(1-e^{t\Delta}\right)
    \mathbb{P}\div\left(\widetilde{F}_{\delta,N} \otimes \widetilde{F}_{\delta,N}\right).
\end{align}
By the direct calculation (see \cite{Fujii-pre}*{Lemma 2.1} for details), there holds
\begin{align}\label{direct}
    \begin{split}
    \Delta_j
    u_{\delta,N}^{(2,1)}(t)
    &={}
    -
    M^2\frac{\delta^2}{2N}
    \Delta_j
    (-\Delta)^{-1}
    \left(1-e^{t\Delta}\right)
    \mathbb{P}
    \begin{pmatrix}
        0 \\ \partial_{x_2}(\Psi^2)
    \end{pmatrix}\\
    &\qquad 
    -
    \frac{\delta^2}{2N}
    \Delta_j
    (-\Delta)^{-1}
    \left(1-e^{t\Delta}\right)
    \mathbb{P}
    \div\left(\nabla^{\perp}\Psi \otimes \nabla^{\perp}\Psi\right)\\
    &=:{}
    \Delta_j
    u_{\delta,N}^{(2,1,1)}(t)
    +
    \Delta_j
    u_{\delta,N}^{(2,1,2)}(t)
    \end{split}
\end{align}
for $j \in \mathbb{Z}$ with $j \leqslant 0$.
Let
\begin{align}\label{Aj}
    A_j
    :=
    \left\{
    \xi \in \mathbb{R}^2 \ ; \ 
    2^{j-1} \leqslant |\xi| \leqslant 2^{j+1},\
    \frac{|\xi|}{2} \leqslant |\xi_2| \leqslant \frac{|\xi|}{\sqrt{2}} 
    \right\}.
\end{align}
The Fourier transform of $u_{\delta,N}^{(2,1,1)}(t)$ is estimated as
\begin{align}\label{est_below-1}
    \left|
    \mathscr{F}
    \left[
    \Delta_j
    u_{\delta,N}^{(2,1,1)}(t)
    \right](\xi)
    \right|
    \geqslant{}&
    \left|
    \mathscr{F}
    \left[
    \Delta_j
    \left(u_{\delta,N}^{(2,1,1)}(t)\right)_2
    \right](\xi)
    \right|\\
    ={}&
    M^2
    \frac{\delta^2}{2N}
    \frac{1-e^{-t|\xi|^2}}{|\xi|^2}
    \left(
    1-\frac{\xi_2^2}{|\xi|^2}
    \right)
    |\xi_2|
    \widehat{\Phi_0}(2^{-j}\xi)
    \left(\widehat{\Psi}*\widehat{\Psi}\right)(\xi)\\
    \geqslant{}&
    c
    M^2 
    \frac{\delta^2}{N}
    \cdot 
    \frac{1-e^{-\frac{1}{4}t2^{2j}}}{2^j}
    \widehat{\Phi_0}(2^{-j}\xi)
    \left(\widehat{\Psi}*\widehat{\Psi}\right)(\xi)
\end{align}
for $\xi \in A_j$,
where $(u_{\delta,N}^{(2,1,1)}(t))_2$ denotes the second component of $u_{\delta,N}^{(2,1,1)}(t)$.
Thus, it holds by the Bernstein inequality and the Plancherel theorem that
\begin{align}\label{est_below-2}
    \begin{split}
    &
    2^{(\frac{2}{p}-1)j}
    \left\| \Delta_j u_{\delta,N}^{(2,1,1)}(T_N) \right\|_{L^p}\\
    &\quad\geqslant{}
    c
    \left\|
    \mathscr{F}
    \left[
    \Delta_j
    u_{\delta,N}^{(2,1,1)}(T_N)
    \right]
    \right\|_{L^2}\\
    &\quad\geqslant{}
    c
    M^2 
    \frac{\delta^2}{N}
    \cdot 
    \frac{1-e^{-\frac{1}{4}T_N2^{2j}}}{2^j}
    \left\| 
    \widehat{\Phi_0}(2^{-j}\xi)
    \left(\widehat{\Psi}*\widehat{\Psi}\right)(\xi)
    \right\|_{L^2_{\xi}(A_j)}\\
    &\quad={}
    c
    M^2 
    \frac{\delta^2}{N}
    \left(1-e^{-2^{2(N+j-1)}}\right)
    \left\| 
    \widehat{\Phi_0}(\eta)
    \left(\widehat{\Psi}*\widehat{\Psi}\right)(2^j\eta)
    \right\|_{L^2_{\eta}(A_0)}.
    \end{split}
\end{align}
Here, we have changed the variables $\eta=2^{-j}\xi$ in the last line of \eqref{est_below-2}.
Since $\widehat{\Psi}(2^j\eta-\mu)=\widehat{\Psi}(\mu)=1$ for all $\eta \in A_0$, $\mu$ with $|\mu| \leqslant 1/2$ and $j \leqslant -2$, we have
\begin{align}\label{est_below-3}
    \left(\widehat{\Psi}*\widehat{\Psi}\right)(2^j\eta)
    =
    \int_{\mathbb{R}^2}
    \widehat{\Psi}(2^j\eta-\mu)\widehat{\Psi}(\mu)d\mu
    \geqslant
    \int_{|\mu| \leqslant \frac{1}{2}}
    d\mu=c>0
\end{align}
for $j \leqslant -2$, which implies
\begin{align}\label{est_below-4}
    \inf_{j \leqslant -2}
    \left\| 
    \widehat{\Phi_0}(\eta)
    \left(\widehat{\Psi}*\widehat{\Psi}\right)(2^j\eta)
    \right\|_{L^2_{\eta}(A_0)}
    \geqslant
    c
    \left\| 
    \widehat{\Phi_0}
    \right\|_{L^2(A_0)}>0.
\end{align}
Hence, we obtain by \eqref{est_below-2} and \eqref{est_below-4} that
\begin{align}\label{est-u(211)-below}
    \begin{split}
    \left\| u_{\delta,N}^{(2,1,1)}(T_N) \right\|_{\dB_{p,1}^{\frac{2}{p}-1}}
    \geqslant{}&
    \sum_{-N \leqslant j \leqslant -2}
    2^{(\frac{2}{p}-1)j}
    \left\|
    \Delta_j
    u_{\delta,N}^{(2,1,1)}(T_N)
    \right\|_{L^p}\\
    \geqslant{}&
    cM^2\frac{\delta^2}{N}\sum_{-N \leqslant j \leqslant -2}\left(1-e^{-2^{2(N+j-1)}}\right)
    \left\| 
    \widehat{\Phi_0}
    \right\|_{L^2(A_0)}\\
    \geqslant{}&
    c_0M^2\delta^2
    \end{split}
\end{align}
for some positive constant $c_0=c_0(p,\Psi)$.
For the estimate of $u_{\delta,N}^{(2,1,2)}$, 
using 
\begin{align}
    u_{\delta,N}^{(2,1,2)}(t)
    =
    -
    \frac{\delta^2}{2N}
    (-\Delta)^{-1}
    \left(1-e^{t\Delta}\right)
    \mathbb{P}
    \div\left(\nabla^{\perp}\Psi \otimes \nabla^{\perp}\Psi\right)
    =
    \frac{\delta^2}{2N}
    \mathcal{D}\left[\nabla^{\perp}\Psi,\nabla^{\perp}\Psi\right](t)
\end{align}
Lemma \ref{lemm:ill-duhamel-est},
we have
\begin{align}\label{est-u(212)}
    \left\| u_{\delta,N}^{(2,1,2)}\right\|_{\widetilde{L^{\infty}}(0,T_N;\dB_{p,1}^{\frac{2}{p}-1})}
    \leqslant{}&
    C
    \delta^2
    \n{ \nabla^{\perp} \Psi}_{\widetilde{L^N}(0,T_N;\dB_{p,1}^{\frac{2}{p}-1+\frac{2}{N}})}^2
    +
    C 
    \frac{\delta^2}{N}
    \n{ \nabla^{\perp} \Psi}_{\dB_{p,1}^{\frac{2}{p}-1}}^2\\
    \leqslant{}&
    C
    \delta^2
    T_N^{\frac{1}{N}}
    \| \Psi \|_{\dB_{p,1}^{\frac{2}{p}+\frac{2}{N}}}^2
    +
    C 
    \frac{\delta^2}{N}
    \| \Psi \|_{\dB_{p,1}^{\frac{2}{p}}}^2\\
    \leqslant{}&
    C_0\delta^2.
\end{align}
for some positive constant $C_0=C_0(p,\Psi)$.
For the estimate of $u_{\delta,N}^{(2,2)}$, 
using Lemma \ref{lemm:duhamel-est-1}, we have
\begin{align}\label{est-u(22)}
    \begin{split}
    \left\| u_{\delta,N}^{(2,2)}\right\|_{\widetilde{L^{\infty}}(0,\infty;\dB_{p,1}^{\frac{2}{p}-1})}
    \leqslant{}&
    C\| \widetilde{F}_{\delta,N} \|_{\dB_{p,1}^{\frac{2}{p}-1}}\left\| e^{t\Delta}\widetilde{F}_{\delta,N} \right\|_{\widetilde{L^2}(0,\infty;\dB_{p,1}^{\frac{2}{p}})}
    +
    C\left\| e^{t\Delta}\widetilde{F}_{\delta,N} \right\|_{\widetilde{L^2}(0,\infty;\dB_{p,1}^{\frac{2}{p}})}^2\\
    \leqslant{}&
    C\| \widetilde{F}_{\delta,N} \|_{\dB_{p,1}^{\frac{2}{p}-1}}^2\\
    \leqslant{}&
    CM^{\frac{4}{p}}\frac{\delta^2}{N}.
    \end{split}
\end{align}
We now fix $M$ so that
\begin{align}
    M:=
    \max
    \left\{
    10, 
    \sqrt{2+\frac{C_0}{c_0}}
    \right\}.
\end{align}
Then, we obtain by \eqref{est-u(211)-below}, \eqref{est-u(212)}, and \eqref{est-u(22)} that
\begin{align}\label{est-u(2)-below}
    \begin{split}
    \left\| u_{\delta,N}^{(2)}(T_N) \right\|_{\dB_{p,1}^{\frac{2}{p}-1}}
    \geqslant{}&
    \left\| u_{\delta,N}^{(2,1,1)}(T_N) \right\|_{\dB_{p,1}^{\frac{2}{p}-1}}\\
    &
    -\left\| u_{\delta,N}^{(2,1,2)}\right\|_{\widetilde{L^{\infty}}(0,T_N;\dB_{p,1}^{\frac{2}{p}-1})}
    -\left\| u_{\delta,N}^{(2,2)}\right\|_{\widetilde{L^{\infty}}(0,\infty;\dB_{p,1}^{\frac{2}{p}-1})}\\
    \geqslant{}&
    \left(M^2c_0-C_0-C\frac{M^{\frac{4}{p}}}{N}\right)\delta^2\\
    \geqslant{}&
    \left(2c_0-\frac{C}{N}\right)\delta^2.
    \end{split}
\end{align}
On the other hand, it follows from Lemma \ref{lemm:ill-duhamel-est}, \eqref{est-u(1)-1}, and \eqref{est-u(1)-2} that
\begin{align}\label{est-u(2)-1}
    \begin{split}
    \left\| u_{\delta,N}^{(2)} \right\|_{\widetilde{L^{\infty}}(0,\infty;\dB_{p,1}^{\frac{2}{p}-1})}
    \leqslant{}&
    CN
    \left\| u_{\delta,N}^{(1)} \right\|_{\widetilde{L^N}(0,T_N;\dB_{p,2}^{\frac{2}{p}-1+\frac{2}{N}})}^2
    +
    C
    \left\| u_{\delta,N}^{(1)} \right\|_{\widetilde{L^{\infty}}(0,\infty;\dB_{p,1}^{\frac{2}{p}-1})}^2\\
    \leqslant{}&
    C\delta^2+C\frac{\delta^2}{N}\\
    \leqslant{}&
    C\delta^2
    \end{split}
\end{align}
and
\begin{align}\label{est-u(2)-2}
    \left\| u_{\delta,N}^{(2)} \right\|_{\widetilde{L^N}(0,T_N;\dB_{p,2}^{\frac{2}{p}-1+\frac{2}{N}})}
    \leqslant{}
    C\sqrt{N}
    \left\| u_{\delta,N}^{(1)} \right\|_{\widetilde{L^N}(0,T_N;\dB_{p,2}^{\frac{2}{p}-1+\frac{2}{N}})}^2
    \leqslant{}
    C\frac{\delta^2}{\sqrt{N}}.
\end{align}


\noindent
{\it Step.4 The construction and estimates for the remainder part.}
To construct a solution to \eqref{eq:nNS-0} with the external force $F_{\delta,N}$, we focus on the perturbation of a solution to \eqref{eq:nNS-0} with the external force $F_{\delta,N}$ from the second approximation $u_{\delta,N}^{(1)}+u_{\delta,N}^{(2)}$.
If $u_{\delta,N}$ is a solution to \eqref{eq:nNS-0} with the external force $F_{\delta,N}$, then $w_{\delta,N}:=u_{\delta,N}-u_{\delta,N}^{(1)}-u_{\delta,N}^{(2)}$ should satisfy
\begin{align}\label{eq:wN}
    \begin{dcases}
    \begin{aligned}
    \partial_t w_{\delta,N} - \Delta w_{\delta,N} 
    +{}
    &
    \mathbb{P} \div 
    \left(
    u_{\delta,N}^{(1)} \otimes u_{\delta,N}^{(2)}
    +
    u_{\delta,N}^{(2)} \otimes u_{\delta,N}^{(1)}
    +
    u_{\delta,N}^{(2)} \otimes u_{\delta,N}^{(2)}\right.\\
    &\quad
    +
    u_{\delta,N}^{(1)} \otimes w_{\delta,N}
    +
    u_{\delta,N}^{(2)} \otimes w_{\delta,N}\\
    &\quad
    +
    w_{\delta,N} \otimes u_{\delta,N}^{(1)}
    +
    \left.
    w_{\delta,N} \otimes u_{\delta,N}^{(2)}
    +
    w_{\delta,N} \otimes w_{\delta,N} 
    \right)
    =0,
    \end{aligned}\\
    \div w_{\delta,N} = 0,\\
    w_{\delta,N}(0,x) = 0.
    \end{dcases}
\end{align}
To construct the mild solution to \eqref{eq:wN}, we consider the map
\begin{align}\label{dfSN}
    \begin{split}
    \mathcal{S}_N[w]:={}
    &
    \mathcal{D}\left[u_{\delta,N}^{(1)} , u_{\delta,N}^{(2)}\right]
    +
    \mathcal{D}\left[u_{\delta,N}^{(2)} , u_{\delta,N}^{(1)}\right]
    +
    \mathcal{D}\left[u_{\delta,N}^{(2)} , u_{\delta,N}^{(2)}\right]\\
    &
    +
    \mathcal{D}\left[u_{\delta,N}^{(1)} , w\right]
    +
    \mathcal{D}\left[u_{\delta,N}^{(2)} , w\right]
    +
    \mathcal{D}\left[w , u_{\delta,N}^{(1)}\right]
    +
    \mathcal{D}\left[w , u_{\delta,N}^{(2)}\right]\\
    &
    +
    \mathcal{D}[w , w] .
    \end{split}
\end{align}
Here, we consider the estimates for the first three terms of the right hand side of \eqref{dfSN}.
By virtue of Lemma \ref{lemm:ill-duhamel-est}, \eqref{est-u(1)-1}, \eqref{est-u(1)-2}, \eqref{est-u(2)-1}, and \eqref{est-u(2)-2}, we have
\begin{align}
    &
    \left\| \mathcal{D}\left[u_{\delta,N}^{(1)} , u_{\delta,N}^{(2)}\right] \right\|_{\widetilde{L^{\infty}}(0,T_N;\dB_{p,1}^{\frac{2}{p}-1})}
    +
    \left\| \mathcal{D}\left[u_{\delta,N}^{(2)} , u_{\delta,N}^{(1)}\right] \right\|_{\widetilde{L^{\infty}}(0,T_N;\dB_{p,1}^{\frac{2}{p}-1})}
    \\
    &\quad 
    \leqslant
    CN
    \left\| u_{\delta,N}^{(1)} \right\|_{\widetilde{L^N}(0,T_N;\dB_{p,2}^{\frac{2}{p}-1+\frac{2}{N}})}
    \left\| u_{\delta,N}^{(2)} \right\|_{\widetilde{L^N}(0,T_N;\dB_{p,2}^{\frac{2}{p}-1+\frac{2}{N}})}\\
    &\qquad
    +
    C
    \left\| u_{\delta,N}^{(1)} \right\|_{\widetilde{L^{\infty}}(0,\infty;\dB_{p,1}^{\frac{2}{p}-1})}
    \left\| u_{\delta,N}^{(2)} \right\|_{\widetilde{L^{\infty}}(0,\infty;\dB_{p,1}^{\frac{2}{p}-1})}\\
    &\quad 
    \leqslant
    C\delta^3,\\
    &
    \begin{aligned}
    \left\| \mathcal{D}\left[u_{\delta,N}^{(2)} , u_{\delta,N}^{(2)}\right] \right\|_{\widetilde{L^{\infty}}(0,T_N;\dB_{p,1}^{\frac{2}{p}-1})}
    &\leqslant 
    CN
    \left\| u_{\delta,N}^{(2)} \right\|_{\widetilde{L^N}(0,T_N;\dB_{p,2}^{\frac{2}{p}-1+\frac{2}{N}})}^2
    +
    C
    \left\| u_{\delta,N}^{(2)} \right\|_{\widetilde{L^{\infty}}(0,\infty;\dB_{p,1}^{\frac{2}{p}-1})}^2\\
    &
    \leqslant
    C\delta^4
    \end{aligned}
\end{align}
and
\begin{align}
    &
    \left\| \mathcal{D}\left[u_{\delta,N}^{(1)} , u_{\delta,N}^{(2)}\right] \right\|_{\widetilde{L^N}(0,T_N;\dB_{p,2}^{\frac{2}{p}-1+\frac{2}{N}})}
    +
    \left\| \mathcal{D}\left[u_{\delta,N}^{(2)} , u_{\delta,N}^{(1)}\right] \right\|_{\widetilde{L^N}(0,T_N;\dB_{p,2}^{\frac{2}{p}-1+\frac{2}{N}})}\\
    &\quad 
    \leqslant{}
    C\sqrt{N}
    \left\| u_{\delta,N}^{(1)} \right\|_{\widetilde{L^N}(0,T_N;\dB_{p,2}^{\frac{2}{p}-1+\frac{2}{N}})}
    \left\| u_{\delta,N}^{(2)} \right\|_{\widetilde{L^N}(0,T_N;\dB_{p,2}^{\frac{2}{p}-1+\frac{2}{N}})}\\
    &\quad
    \leqslant{}
    C\frac{\delta^3}{\sqrt{N}},\\
    &
    \begin{aligned}
    \left\| \mathcal{D}\left[u_{\delta,N}^{(2)} , u_{\delta,N}^{(2)}\right] \right\|_{\widetilde{L^N}(0,T_N;\dB_{p,2}^{\frac{2}{p}-1+\frac{2}{N}})}
    &\leqslant{}
    C\sqrt{N}
    \left\| u_{\delta,N}^{(2)} \right\|_{\widetilde{L^N}(0,T_N;\dB_{p,2}^{\frac{2}{p}-1+\frac{2}{N}})}^2\\
    &
    \leqslant{}
    C\frac{\delta^4}{\sqrt{N}}.
    \end{aligned}
\end{align}
Therefore, there exists a positive constant $C_1=C_1(p,\Psi)$ such that 
\begin{align}
    &\left\| \mathcal{D}\left[u_{\delta,N}^{(1)} , u_{\delta,N}^{(2)}\right] + \mathcal{D}\left[u_{\delta,N}^{(2)} , u_{\delta,N}^{(1)}\right] + \mathcal{D}\left[u_{\delta,N}^{(2)} , u_{\delta,N}^{(2)}\right] \right\|_{\widetilde{L^{\infty}}(0,T_N;\dB_{p,1}^{\frac{2}{p}-1})}
    \leqslant
    C_1\delta^3,\\
    &
    \left\| \mathcal{D}\left[u_{\delta,N}^{(1)} , u_{\delta,N}^{(2)}\right] + \mathcal{D}\left[u_{\delta,N}^{(2)} , u_{\delta,N}^{(1)}\right] + \mathcal{D}\left[u_{\delta,N}^{(2)} , u_{\delta,N}^{(2)}\right] \right\|_{\widetilde{L^N}(0,T_N;\dB_{p,2}^{\frac{2}{p}-1+\frac{2}{N}})}
    \leqslant
    C_1 \frac{\delta^3}{\sqrt{N}}.
\end{align}
Now, we shall show that $\mathcal{S}_N[\cdot]$ is a contraction map on the complete metric space $(X_N,d_{X_N})$ defined by
\begin{align}
    &
    X_N
    :=
    \left\{
    \begin{aligned}
    w 
    \in{}& 
    \widetilde{C}([0,T_N]; \dB_{p,1}^{\frac{2}{p}-1}(\mathbb{R}^2)) \\ 
    &\cap 
    \widetilde{L^N}(0,T_N;\dB_{p,2}^{\frac{2}{p}-1+\frac{2}{N}}(\mathbb{R}^2))
    \end{aligned}\ ;\ 
    \begin{aligned}
    &\| w \|_{\widetilde{L^{\infty}}(0,T_N;\dB_{p,1}^{\frac{2}{p}-1})}
    \leqslant
    2C_1\delta^3,\\
    &\| w \|_{\widetilde{L^N}(0,T_N;\dB_{p,2}^{\frac{2}{p}-1+\frac{2}{N}})}
    \leqslant
    2C_1\frac{\delta^3}{\sqrt{N}}.
    \end{aligned}
    \right\},\\
    &
    d_{X_N}(w_1,w_2)
    :=
    \| w_1 - w_2 \|_{\widetilde{L^{\infty}}(0,T_N;\dB_{p,1}^{\frac{2}{p}-1})}
    +
    \sqrt{N}
    \| w_1 - w_2 \|_{\widetilde{L^N}(0,T_N;\dB_{p,2}^{\frac{2}{p}-1+\frac{2}{N}})}.
\end{align}
Let $w \in X_N$.
Then, it follows from Lemma \ref{lemm:ill-duhamel-est}, \eqref{est-u(1)-1}, \eqref{est-u(1)-2}, \eqref{est-u(2)-1}, and \eqref{est-u(2)-2}, that
\begin{align}
    &\left\| \mathcal{S}_N[w] \right\|_{\widetilde{L^{\infty}}(0,T_N;\dB_{p,1}^{\frac{2}{p}-1})}\\
    &\quad
    \begin{aligned}
    \leqslant{}
    C_1\delta^3
    +
    C\sum_{k=1}^2 
    &
    \left(
    N
    \left\| u_{\delta,N}^{(k)} \right\|_{\widetilde{L^N}(0,T_N;\dB_{p,2}^{\frac{2}{p}-1+\frac{2}{N}})}
    \| w \|_{\widetilde{L^N}(0,T_N;\dB_{p,2}^{\frac{2}{p}-1+\frac{2}{N}})}\right.\\
    &\quad 
    \left.+
    \left\| u_{\delta,N}^{(k)} \right\|_{\widetilde{L^{\infty}}(0,T_N;\dB_{p,1}^{\frac{2}{p}-1})}
    \| w \|_{\widetilde{L^{\infty}}(0,T_N;\dB_{p,1}^{\frac{2}{p}-1})}
    \right)
    \end{aligned}\\
    &\qquad 
    +
    CN\| w \|_{\widetilde{L^N}(0,T_N;\dB_{p,2}^{\frac{2}{p}-1+\frac{2}{N}})}^2
    +
    C\| w \|_{\widetilde{L^{\infty}}(0,T_N;\dB_{p,1}^{\frac{2}{p}-1})}^2\\
    &\quad 
    \leqslant
    C_1\delta^3
    +
    C\delta\sqrt{N}
    \| w \|_{\widetilde{L^N}(0,T_N;\dB_{p,2}^{\frac{2}{p}-1+\frac{2}{N}})}
    +
    C\delta
    \| w \|_{\widetilde{L^{\infty}}(0,T_N;\dB_{p,1}^{\frac{2}{p}-1})}\\
    &\qquad
    +
    CN\| w \|_{\widetilde{L^N}(0,T_N;\dB_{p,2}^{\frac{2}{p}-1+\frac{2}{N}})}^2
    +
    C\| w \|_{\widetilde{L^{\infty}}(0,T_N;\dB_{p,1}^{\frac{2}{p}-1})}^2\\
    &\quad
    \leqslant
    C_1\delta^3
    +
    C_2\delta^4
\end{align}
and
\begin{align}
    &\| \mathcal{S}_N[w] \|_{\widetilde{L^N}(0,T_N;\dB_{p,2}^{\frac{2}{p}-1+\frac{2}{N}})}\\
    &\quad\leqslant{}
    C_1\frac{\delta^3}{\sqrt{N}}
    +
    C\sqrt{N}
    \sum_{k=1}^2
    \left\| u_{\delta,N}^{(k)} \right\|_{\widetilde{L^N}(0,T_N;\dB_{p,2}^{\frac{2}{p}-1+\frac{2}{N}})}
    \| w \|_{\widetilde{L^N}(0,T_N;\dB_{p,2}^{\frac{2}{p}-1+\frac{2}{N}})}\\
    &\qquad
    +
    C
    \sqrt{N}
    \| w \|_{\widetilde{L^N}(0,T_N;\dB_{p,2}^{\frac{2}{p}-1+\frac{2}{N}})}^2\\
    &\quad\leqslant{}
    C_1\frac{\delta^3}{\sqrt{N}}
    +
    C\delta
    \| w \|_{\widetilde{L^N}(0,T_N;\dB_{p,2}^{\frac{2}{p}-1+\frac{2}{N}})}
    +
    C
    \sqrt{N}
    \| w \|_{\widetilde{L^N}(0,T_N;\dB_{p,2}^{\frac{2}{p}-1+\frac{2}{N}})}^2\\
    &\quad\leqslant{}
    C_1\frac{\delta^3}{\sqrt{N}}
    +
    C_2\delta\frac{\delta^3}{\sqrt{N}}
\end{align}
for some positive constant $C_2=C_2(p,\Psi)$.
Let $w_1,w_2 \in X_N$.
Then since
\begin{align}
    \mathcal{S}_N[w_1] - \mathcal{S}_N[w_2]
    ={}
    &
    \mathcal{D}\left[u_{\delta,N}^{(1)} , w_1-w_2\right]
    +
    \mathcal{D}\left[u_{\delta,N}^{(2)} , w_1-w_2\right]\\
    &
    +
    \mathcal{D}\left[w_1-w_2 , u_{\delta,N}^{(1)}\right]
    +
    \mathcal{D}\left[w_1-w_2 , u_{\delta,N}^{(2)}\right]\\
    &
    +
    \mathcal{D}[w_1 , w_1-w_2]
    +
    \mathcal{D}[w_1-w_2 , w_2],
\end{align}
we see by Lemma \ref{lemm:ill-duhamel-est} that
\begin{align}
    &\left\| \mathcal{S}_N[w_1] - \mathcal{S}_N[w_2] \right\|_{\widetilde{L^{\infty}}(0,T_N;\dB_{p,1}^{\frac{2}{p}-1})}\\
    &\quad
    \begin{aligned}
    \leqslant{}
    C\sum_{k=1}^2 
    &
    \left(
    N
    \left\| u_{\delta,N}^{(k)} \right\|_{\widetilde{L^N}(0,T_N;\dB_{p,2}^{\frac{2}{p}-1+\frac{2}{N}})}
    \left\| \mathcal{S}_N[w_1] - \mathcal{S}_N[w_2] \right\|_{\widetilde{L^N}(0,T_N;\dB_{p,2}^{\frac{2}{p}-1+\frac{2}{N}})}\right.\\
    &\quad 
    \left.+
    \left\| u_{\delta,N}^{(k)} \right\|_{\widetilde{L^{\infty}}(0,T_N;\dB_{p,1}^{\frac{2}{p}-1})}
    \left\| \mathcal{S}_N[w_1] - \mathcal{S}_N[w_2] \right\|_{\widetilde{L^{\infty}}(0,T_N;\dB_{p,1}^{\frac{2}{p}-1})}
    \right)
    \end{aligned}\\
    &\qquad
    \begin{aligned}
    +
    C\sum_{k=1}^2 
    &
    \left(
    N
    \left\| w_k\right\|_{\widetilde{L^N}(0,T_N;\dB_{p,2}^{\frac{2}{p}-1+\frac{2}{N}})}
    \left\| \mathcal{S}_N[w_1] - \mathcal{S}_N[w_2] \right\|_{\widetilde{L^N}(0,T_N;\dB_{p,2}^{\frac{2}{p}-1+\frac{2}{N}})}\right.\\
    &\quad 
    \left.+
    \left\| w_k \right\|_{\widetilde{L^{\infty}}(0,T_N;\dB_{p,1}^{\frac{2}{p}-1})}
    \left\| \mathcal{S}_N[w_1] - \mathcal{S}_N[w_2] \right\|_{\widetilde{L^{\infty}}(0,T_N;\dB_{p,1}^{\frac{2}{p}-1})}
    \right)
    \end{aligned}\\\\
    &\quad 
    \leqslant
    C_3\delta\sqrt{N}
    \| w_1-w_2 \|_{\widetilde{L^N}(0,T_N;\dB_{p,2}^{\frac{2}{p}-1+\frac{2}{N}})}
    +
    C_3\delta
    \| w_1-w_2 \|_{\widetilde{L^{\infty}}(0,T_N;\dB_{p,1}^{\frac{2}{p}-1})}
\end{align}
and
\begin{align}
    &\| \mathcal{S}_N[w_1]-\mathcal{S}_N[w_2] \|_{\widetilde{L^N}(0,T_N;\dB_{p,2}^{\frac{2}{p}-1+\frac{2}{N}})}\\
    &\quad\leqslant{}
    C\sqrt{N}
    \sum_{k=1}^2
    \left\| u_{\delta,N}^{(k)} \right\|_{\widetilde{L^N}(0,T_N;\dB_{p,2}^{\frac{2}{p}-1+\frac{2}{N}})}
    \| w_1-w_2 \|_{\widetilde{L^N}(0,T_N;\dB_{p,2}^{\frac{2}{p}-1+\frac{2}{N}})}\\
    &\qquad
    +
    C
    \sqrt{N}
    \sum_{k=1}^2
    \left\| w_k \right\|_{\widetilde{L^N}(0,T_N;\dB_{p,2}^{\frac{2}{p}-1+\frac{2}{N}})}
    \| w_1-w_2 \|_{\widetilde{L^N}(0,T_N;\dB_{p,2}^{\frac{2}{p}-1+\frac{2}{N}})}\\
    &\quad\leqslant{}
    C_3
    \delta
    \| w_1-w_2 \|_{\widetilde{L^N}(0,T_N;\dB_{p,2}^{\frac{2}{p}-1+\frac{2}{N}})}
\end{align}
for some positive constant $C_3=C_3(p,\Psi)$.
Here, we choose $\delta$ so small that
\begin{align}
    0 < \delta \leqslant \delta_1:=\min\left\{ \frac{C_1}{C_2}, \frac{1}{4C_3},\frac{c_0}{2C_1}\right\}. 
\end{align}
Then, we have
\begin{align}
    &
    \| \mathcal{S}_N[w] \|_{\widetilde{L^{\infty}}(0,T_N;\dB_{p,1}^{\frac{2}{p}-1})}
    \leqslant
    2C_1\delta^3,\\
    &
    \| \mathcal{S}_N[w] \|_{\widetilde{L^N}(0,T_N;\dB_{p,2}^{\frac{2}{p}-1+\frac{2}{N}})}
    \leqslant
    2C_1\frac{\delta^3}{\sqrt{N}},\\
    &
    d_{X_N}(\mathcal{S}_N[w_1],\mathcal{S}_N[w_2])
    \leqslant
    \frac{1}{2}
    d_{X_N}(w_1,w_2),
\end{align}
which implies that $\mathcal{S}_N[\cdot]$ is a contraction map on $(X_N,d_{X_N})$.
Hence, by the Banach fixed point theorem, there exists a unique element $w_{\delta,N} \in X_N$ such that $w_{\delta,N}=\mathcal{S}_N[w_{\delta,N}]$, which means that the mild solution $w_{\delta,N}$ of \eqref{eq:wN} uniquely exists in $X_N$.

\noindent
{\it Step.5 Conclusion.}
We see that the function 
\begin{align}
    u_{\delta,N}:=u_{\delta,N}^{(1)}+u_{\delta,N}^{(2)}+w_{\delta,N} \in \widetilde{C}( [0,T_N] ; \dB_{p,1}^{\frac{2}{p}-1}(\mathbb{R}^2))
\end{align}
is a mild solution to \eqref{eq:nNS-0} with the external force $F_{\delta,N}$ and also obtain by \eqref{est-u(1)-1}, \eqref{est-u(2)-below}, and $w_{\delta,N} \in X_N$ that
\begin{align}
    \| u_{\delta,N}(T_N) \|_{\dB_{p,1}^{\frac{2}{p}-1}}
    \geqslant{}&
    \left\| u_{\delta,N}^{(2)}(T_N) \right\|_{\dB_{p,1}^{\frac{2}{p}-1}}
    -
    \left\| u_{\delta,N}^{(1)} \right\|_{\widetilde{L^{\infty}}(0,\infty;\dB_{p,1}^{\frac{2}{p}-1})}
    -
    \| w_{\delta,N} \|_{\widetilde{L^{\infty}}(0,T_N;\dB_{p,1}^{\frac{2}{p}-1})}\\
    \geqslant {}&
    \left(2c_0-\frac{C}{N}\right)\delta^2 - C\frac{\delta}{\sqrt{N}} - 2C_1\delta^3\\
    \geqslant {}&
    \left(c_0-\frac{C}{N}\right)\delta^2 - C\frac{\delta}{\sqrt{N}},
\end{align}
which yields
\begin{align}
    \liminf_{N \to \infty}
    \| u_{\delta,N}(T_N) \|_{\dB_{p,1}^{\frac{2}{p}-1}}
    \geqslant{}
    c_0 \delta^2.
\end{align}
It follows from \eqref{est-u(1)-1}, \eqref{est-u(2)-1}, and $w \in X_N$ that
\begin{align}
    &\| u_{\delta,N} \|_{\widetilde{L^{\infty}}(0,T_N;\dB_{p,1}^{\frac{2}{p}-1})}\\
    &\quad 
    \leqslant{}
    \left\| u_{\delta,N}^{(1)} \right\|_{\widetilde{L^{\infty}}(0,\infty;\dB_{p,1}^{\frac{2}{p}-1})}
    +
    \left\| u_{\delta,N}^{(2)} \right\|_{\widetilde{L^{\infty}}(0,T_N;\dB_{p,1}^{\frac{2}{p}-1})}
    +
    \| w_{\delta,N} \|_{\widetilde{L^{\infty}}(0,T_N;\dB_{p,1}^{\frac{2}{p}-1})}\\
    &\quad 
    \leqslant{}
    C\frac{\delta}{\sqrt{N}}
    +
    C\delta^2
    +
    C\delta^3,
\end{align}
which implies
\begin{align}
    \limsup_{N\to \infty}
    \| u_{\delta,N} \|_{\widetilde{L^{\infty}}(0,T_N;\dB_{p,1}^{\frac{2}{p}-1})}
    \leqslant
    C\delta^2
    +
    C\delta^3
    \leqslant
    C\delta^2.
\end{align}
Thus, we complete the proof.
\end{proof}

\subsection{Global solutions around the stationary flow}\label{subseq:1}
In contrast to the previous subsection, if we assume that the stationary problem \eqref{eq:sNS-1} possesses a solution $U$ for some external force $F$
and then consider the nonstationary Navier--Stokes equations \eqref{eq:nNS-0} with the same external force $F$ as for $U$.
Under this assumption, we may prove that \eqref{eq:nNS-0} admits a bounded-in-time global solution.
\begin{thm}\label{thm:stability}
Let $1 \leqslant p < 4$ and $1\leqslant q < \infty$.
Then, there exist a positive constant $\delta_2=\delta_2(p,q)$ and an absolute positive constant $K_2$ 
such that 
if a given external force $F \in \dB_{p,q}^{\frac{2}{p}-3}(\mathbb{R}^2)$ generates a solution $U \in \dB_{p,q}^{\frac{2}{p}-1}(\mathbb{R}^2)$ to \eqref{eq:sNS-1} satisfying
\begin{align}
    \| U \|_{\dB_{p,q}^{\frac{2}{p}-1}} \leqslant \delta_2,
\end{align}
then
\eqref{eq:nNS-0} with the same external force $F$ admits a global mild solution $u$ 
in the class
\begin{gather}\label{class-u}
    u \in \widetilde{C}([0,\infty);\dB_{p,q}^{\frac{2}{p}-1}(\mathbb{R}^2)),\qquad
    \|u\|_{ 
    \widetilde{L^{\infty}}(0,\infty;\dB_{p,q}^{\frac{2}{p}-1})}
    \leqslant
    K_2\| U \|_{\dB_{p,q}^{\frac{2}{p}-1}}.
\end{gather}
\end{thm}

Assuming the existence of the stationary solution, we consider the perturbation $v=u-U$, which should solve
\begin{align}\label{eq:nNS-2}
    \begin{cases}
    \partial_t v - \Delta v + \mathbb{P}\div( U \otimes v + v \otimes U + v \otimes v ) = 0, & \quad t>0,x \in \mathbb{R}^2,\\
    \div v = 0, & \quad t \geqslant 0,x \in \mathbb{R}^2,\\
    v(0,x) = - U(x), & \quad x \in \mathbb{R}^2,
    \end{cases}
\end{align}
then \eqref{eq:nNS-2} possesses no external force that does not decay as $t \to \infty$, which implies that the solution $v$ of \eqref{eq:nNS-2} is expected to decay as $t \to \infty$ and belong to some time integrable function spaces. 
Since the nonlinear estimate is closed in 
\begin{align}\label{class}
    \widetilde{C}([0,\infty);\dB_{p,q}^{\frac{2}{p}-1}(\mathbb{R}^2))\cap \widetilde{L^r}(0,\infty;\dB_{p,q}^{\frac{2}{p}-1+\frac{2}{r}}(\mathbb{R}^2))
\end{align}
for some $2 < r < \infty$ (see Lemma \ref{lemm:duhamel-est-1}), we may establish the global solution $v$ to \eqref{eq:nNS-2} in the class \eqref{class}.
We then obtain the desired solution by $u:=v+U$.

Now, we provide the precise proof as follows.
\begin{proof}[Proof of Theorem \ref{thm:stability}]
We first construct a mild solution $v$ of \eqref{eq:nNS-2} solving the following integral equation:
\begin{align}\label{int:nNS-2}
    v(t) 
    =
    -e^{t\Delta}U
    +
    \mathcal{D}[U,v](t)
    +
    \mathcal{D}[v,U](t)
    +
    \mathcal{D}[v,v](t),
\end{align}
where the nonlinear term $\mathcal{D}[\cdot , \cdot ]$ is defined in
\eqref{df-duhamel}.
To this end, we focus on the map 
\begin{align}
    \mathcal{S}[v](t)
    :=
    -e^{t\Delta}U
    +
    \mathcal{D}[U,v](t)
    +
    \mathcal{D}[v,U](t)
    +
    \mathcal{D}[v,v](t)
\end{align}
and shall show that $\mathcal{S}[\cdot]$ is a contraction map on the complete metric space $(X,d_X)$ defined by
\begin{align}
    &X:={}
    \left\{
    \begin{aligned}
    &
    v 
    \in 
    \widetilde{C}([0,\infty);\dB_{p,q}^{\frac{2}{p}-1}(\mathbb{R}^2)) 
    \cap 
    \widetilde{L^r}(0,\infty;\dB_{p,q}^{\frac{2}{p}-1+\frac{2}{r}}(\mathbb{R}^2))\ ;\\
    &\qquad \qquad \qquad
    \|v\|_{ 
    \widetilde{L^{\infty}}(0,\infty;\dB_{p,q}^{\frac{2}{p}-1})\cap \widetilde{L^r}(0,\infty;\dB_{p,q}^{\frac{2}{p}-1+\frac{2}{r}})}
    \leqslant
    2C_4\| U \|_{\dB_{p,q}^{\frac{2}{p}-1}}
    \end{aligned}
    \right\},\\
    &d_X(v_1,v_2)
    :={}
    \|v_1-v_2\|_{\widetilde{L^r}(0,\infty;\dB_{p,q}^{\frac{2}{p}-1+\frac{2}{r}})},
\end{align}
where $r=r(p)$ is a fixed exponent satisfying
\begin{align}
    \max\left\{ 0, 1 - \frac{2}{p} \right\}
    < 
    \frac{1}{r}
    < 
    \frac{1}{2}
\end{align}
and the positive constant $C_4$ is determined by the estimate
\begin{align}
    \left\| e^{t\Delta}U \right\|_{\widetilde{L^{\infty}}(0,\infty;\dB_{p,q}^{\frac{2}{p}-1})\cap \widetilde{L^r}(0,\infty;\dB_{p,q}^{\frac{2}{p}-1+\frac{2}{r}})}
    \leqslant
    C_4
    \| U \|_{\dB_{p,q}^{\frac{2}{p}-1}},
\end{align}
which is ensured by Lemma \ref{lemm:max-reg}.
Then, it follows from Lemma \ref{lemm:duhamel-est-1} that
\begin{align}
    &
    \| \mathcal{S}[v] \|_{\widetilde{L^{\infty}}(0,\infty;\dB_{p,q}^{\frac{2}{p}-1}) \cap \widetilde{L^r}(0,\infty;\dB_{p,q}^{\frac{2}{p}-1+\frac{2}{r}})}\\
    &\quad
    \leqslant
    C_4\| U \|_{\dB_{p,q}^{\frac{2}{p}-1}}
    +
    C\| U \|_{\dB_{p,q}^{\frac{2}{p}-1}}\| v \|_{\widetilde{L^r}(0,\infty;\dB_{p,q}^{\frac{2}{p}-1+\frac{2}{r}})}
    +
    C\| v \|_{\widetilde{L^r}(0,\infty;\dB_{p,q}^{\frac{2}{p}-1+\frac{2}{r}})}^2\\
    &\quad
    \leqslant
    C_4\| U \|_{\dB_{p,q}^{\frac{2}{p}-1}}
    +
    C_5\| U \|_{\dB_{p,q}^{\frac{2}{p}-1}}\| v \|_{\widetilde{L^r}(0,\infty;\dB_{p,q}^{\frac{2}{p}-1+\frac{2}{r}})}
\end{align}
for all $v \in X$,  with some positive constant $C_5=C_5(p,q,r)$.
Since there holds
\begin{align}
    \mathcal{S}[v_1] - \mathcal{S}[v_2]
    =
    \mathcal{D}[U,v_1-v_2]
    +
    \mathcal{D}[v_1-v_2,U]
    +
    \mathcal{D}[v_1,v_1-v_2]
    +
    \mathcal{D}[v_1-v_2,v_2],
\end{align}
we have by Lemma \ref{lemm:duhamel-est-1} that
\begin{align}
    \| \mathcal{S}[v_1] - \mathcal{S}[v_2] \|_{\widetilde{L^r}(0,\infty;\dB_{p,q}^{\frac{2}{p}-1+\frac{2}{r}})}
    \leqslant{}&
    C\| U \|_{\dB_{p,q}^{\frac{2}{p}-1}}\| v_1 - v_2 \|_{\widetilde{L^r}(0,\infty;\dB_{p,q}^{\frac{2}{p}-1+\frac{2}{r}})}\\
    &+
    C\sum_{\ell=1}^2\| v_{\ell} \|_{\widetilde{L^r}(0,\infty;\dB_{p,q}^{\frac{2}{p}-1+\frac{2}{r}})}\| v_1-v_2 \|_{\widetilde{L^r}(0,\infty;\dB_{p,q}^{\frac{2}{p}-1+\frac{2}{r}})}\\
    \leqslant{}&
    C_6\| U \|_{\dB_{p,q}^{\frac{2}{p}-1}}\| v_1 - v_2 \|_{\widetilde{L^r}(0,\infty;\dB_{p,q}^{\frac{2}{p}-1+\frac{2}{r}})}
\end{align}
for all $v_1,v_2 \in X$, with some positive constant $C_6=C_6(p,q,r)$.
Now, we assume that the stationary solution $U \in \dB_{p,q}^{\frac{2}{p}-1}(\mathbb{R}^2)$ satisfies
\begin{align}
    \| U \|_{\dB_{p,q}^{\frac{2}{p}-1}}
    \leqslant
    \delta_2:= 
    \min \left\{ \frac{C_4}{C_5}, \frac{1}{2C_6} \right\}.
\end{align}
Then, we obtain 
\begin{align}
    &\| \mathcal{S}[v] \|_{\widetilde{L^{\infty}}(0,\infty;\dB_{p,q}^{\frac{2}{p}-1}) \cap \widetilde{L^r}(0,\infty;\dB_{p,q}^{\frac{2}{p}-1+\frac{2}{r}})}
    \leqslant
    2C_4\| U \|_{\dB_{p,q}^{\frac{2}{p}-1}},\\
    &\| \mathcal{S}[v_1] - \mathcal{S}[v_2] \|_{\widetilde{L^r}(0,\infty;\dB_{p,q}^{\frac{2}{p}-1+\frac{2}{r}})}
    \leqslant
    \frac{1}{2}
    \| v_1 - v_2 \|_{\widetilde{L^r}(0,\infty;\dB_{p,q}^{\frac{2}{p}-1+\frac{2}{r}})}
\end{align}
for all $v,v_1,v_2 \in X$,
which implies $\mathcal{S}[\cdot]$ is a contraction map on $(X,d_X)$.
Hence, the Banach fixed point theorem implies that there exists a unique $v \in X$ such that $v=\mathcal{S}[v]$.

Now, we put $u:=v+U$.
Then, we see that $u$ is a mild solution to \eqref{eq:nNS-0} in the class $\widetilde{C}([0,\infty);\dB_{p,q}^{\frac{2}{p}-1}(\mathbb{R}^2))$, and it holds
\begin{align}
    \| u \|_{\widetilde{L^{\infty}}(0,\infty;\dB_{p,q}^{\frac{2}{p}-1})}
    \leqslant
    \| v \|_{\widetilde{L^{\infty}}(0,\infty;\dB_{p,q}^{\frac{2}{p}-1})}
    +
    \| U \|_{\dB_{p,q}^{\frac{2}{p}-1}}
    \leqslant
    (2C_4+1)\| U \|_{\dB_{p,q}^{\frac{2}{p}-1}}.
\end{align}
Thus, we complete the proof.
\end{proof}

\section{Proof of Theorem \ref{thm:main}}\label{sec:pf}
Now, we are in a position to present the proof of our main result.
\begin{proof}[Proof of Theorem \ref{thm:main}]
Let $\delta_1$ and $\delta_2$ be the positive constants appearing in Theorems \ref{thm:ill} and \ref{thm:stability}, respectively.
Let $K_0$, $K_1$, and $K_2$ be the positive constants appearing in Lemma \ref{lemm:prod-2}, Theorem \ref{thm:ill}, and Theorem \ref{thm:stability}, respectively.
We define 
\begin{align}
    \delta_0:=\min \left\{ \delta_2, \delta_3, \frac{\delta_3^2}{2K_1K_2} \right\},
\end{align}
where $\delta_3$ is a positive constant given by
\begin{align}
    \delta_3
    :=
    \min
    \left\{
    \delta_1,\frac{1}{2K_0(K_1+K_2)}
    \right\}.
\end{align}
We consider the sequence $F_N:=F_{\delta_3,N}$, which is defined in \eqref{df-FN} with $\delta$ replaced by $\delta_3$.
Note that Theorem \ref{thm:ill} yields
\begin{align}
    \| F_N \|_{\dB_{p,1}^{\frac{2}{p}-3}}
    \leqslant 
    \frac{K_1\delta_3}{\sqrt{N}} \to 0
    \qquad
    {\rm \ as\ }N \to \infty.
\end{align}
Let us consider the nonstationary Navier--Stokes equations
\begin{align}\label{eq:nNS-3}
\begin{cases}
    \partial_t u - \Delta u + \mathbb{P}\div(u \otimes u) = \mathbb{P} F_N, \qquad & t>0, x \in \mathbb{R}^2,\\
    \div u = 0, \qquad &  t\geqslant 0, x \in \mathbb{R}^2,\\
    u(0,x)=0, \qquad & x \in \mathbb{R}^2.
\end{cases}
\end{align} 
By Theorem \ref{thm:ill}, 
there exists a $N_0=N_0(p) \in \mathbb{N}$ such that for each $N \in \mathbb{N}$ with $N \geqslant N_0$, \eqref{eq:nNS-3} possesses a solution $u_N = u_{\delta_3,N} \in \widetilde{C}([0,T_N] ; \dB_{p,1}^{\frac{2}{p}-1}(\mathbb{R}^2))$ 
satisfying
\begin{align}\label{uN-condi-2}
    \| u_N(T_N) \|_{\dB_{p,1}^{\frac{2}{p}-1}} \geqslant \frac{\delta_3^2}{K_1},\qquad
    \| u_N \|_{\widetilde{L^{\infty}}(0,T_N;\dB_{p,1}^{\frac{2}{p}-1})}\leqslant K_1\delta_3.
\end{align}
Here, we have set $T_N:=2^{2N}$.

Assume to contrary that there exist an integer $N' \geqslant N_0$ and a solution $ U_{N'}  \in \dB_{p,1}^{\frac{2}{p}-1}(\mathbb{R}^2)$ of \eqref{eq:sNS-1} with the external force  $F_{N'}$ satisfying
\begin{align}\label{U_k-1}
    \| U_{N'} \|_{\dB_{p,1}^{\frac{2}{p}-1}}
    <
    \delta_0.
\end{align}
Then, by \eqref{U_k-1} and Theorem \ref{thm:stability}, 
each
$F_{N'}$ generates a global-in-time solution $\widetilde{u}_{N'} \in \widetilde{C}([0,\infty) ; \dB_{p,1}^{\frac{2}{p}-1}(\mathbb{R}^2))$ to the nonstationary Navier--Stokes equations 
\eqref{eq:nNS-3}
satisfying 
\begin{align}\label{uN-condi-1}
    \left\| \widetilde{u}_{N'} \right\|_{\widetilde{L^{\infty}}(0,\infty;\dB_{p,1}^{\frac{2}{p}-1})}
    \leqslant
    K_2\| U_{N'} \|_{\dB_{p,1}^{\frac{2}{p}-1}}
    \leqslant
    K_2\delta_3.
\end{align}
Next, we show that 
these two solutions $\widetilde{u}_{N'}$ and $u_{N'}$ coincides on $[0,T_{N'}]$.
Since $\widetilde{u}_{N'}-u_{N'}$ enjoys
\begin{align}
    \widetilde{u}_{N'}-u_{N'}
    =
    \mathcal{D}\left[\widetilde{u}_{N'},\widetilde{u}_{N'}-u_{N'}\right]
    +
    \mathcal{D}\left[\widetilde{u}_{N'}-u_{N'}, u_{N'}\right],
\end{align}
we see by Lemma \ref{lemm:prod-2} that
\begin{align}\label{uniq-1}
    \begin{split}
    &\sup_{0\leqslant t \leqslant T_{N'}}\| \widetilde{u}_{N'}(t)-u_{N'}(t) \|_{\dB_{p,\infty}^{\frac{2}{p}-1}}\\
    &\quad
    \leqslant{}
    K_0
    \| u_{N'} \|_{\widetilde{L^{\infty}}(0,T_{N'};\dB_{p,1}^{\frac{2}{p}-1})}
    \sup_{0 \leqslant t \leqslant T_{N'}}
    \| \widetilde{u}_{N'}(t)-u_{N'}(t) \|_{\dB_{p,\infty}^{\frac{2}{p}-1}}\\
    &\qquad
    +
    K_0
    \| \widetilde{u}_{N'} \|_{\widetilde{L^{\infty}}(0,T_{N'};\dB_{p,1}^{\frac{2}{p}-1})}
    \sup_{0 \leqslant t \leqslant T_{N'}}
    \| \widetilde{u}_{N'}(t)-u_{N'}(t) \|_{\dB_{p,\infty}^{\frac{2}{p}-1}}
    \\
    &\quad
    \leqslant{}    
    K_0
    \left(
    K_1
    +
    K_2
    \right)
    \delta_3
    \sup_{0 \leqslant t \leqslant T_{N'}}
    \| \widetilde{u}_{N'}(t)-u_{N'}(t) \|_{\dB_{p,\infty}^{\frac{2}{p}-1}}
    \\
    &\quad
    \leqslant{}
    \frac{1}{2}
    \sup_{0 \leqslant t \leqslant T_{N'}}
    \| \widetilde{u}_{N'}(t)-u_{N'}(t) \|_{\dB_{p,\infty}^{\frac{2}{p}-1}},
    \end{split}
\end{align}
which implies
\begin{align}\label{uniq}
    \widetilde{u}_{N'}(t)=u_{N'}(t) \quad {\rm for\ all\ }0 \leqslant t \leqslant T_{N'}.
\end{align}
Hence, it follows from \eqref{uN-condi-2}, \eqref{uN-condi-1}, and \eqref{uniq} that
\begin{align}
    \| U_{N'} \|_{\dB_{p,1}^{\frac{2}{p}-1}}
    \geqslant{}&
    \frac{1}{K_2}
    \| \widetilde{u}_{N'} \|_{\widetilde{L^{\infty}}(0,T_{N'};\dB_{p,1}^{\frac{2}{p}-1})}\\
    \geqslant{}&
    \frac{1}{K_2}
    \| \widetilde{u}_{N'}(T_{N'}) \|_{\dB_{p,1}^{\frac{2}{p}-1}}\\
    ={}& 
    \frac{1}{K_2}
    \| u_{N'}(T_{N'}) \|_{\dB_{p,1}^{\frac{2}{p}-1}}\\
    \geqslant{}&
    \frac{\delta_3^2}{K_1K_2}\\ 
    \geqslant{}&
    2\delta_0,
\end{align}
which contradicts \eqref{U_k-1}.
Thus, we complete the proof.
\end{proof}


\noindent
{\bf Conflict of interest statement.}\\
The author has declared no conflicts of interest.

\noindent
{\bf Acknowledgements.} \\
The author was supported by Grant-in-Aid for JSPS Research Fellow, Grant Number JP20J20941.
The author would like to express his sincere gratitude to Professor Keiichi Watanabe for many valuable comments on Section \ref{sec:intro}.

\end{document}